\documentclass[11pt]{amsart}

\usepackage{amsmath,amssymb,amscd,amsthm,amsxtra,amsfonts}
\usepackage[dvips]{graphics,epsfig}
\usepackage[all]{xy}
\usepackage{url}

\usepackage{color}

\allowdisplaybreaks[2]

\newtheorem{theorem}{Theorem} [section]

\newtheorem{lemma}[theorem]{Lemma}
\newtheorem{proposition}[theorem]{Proposition}
\newtheorem{remark}[theorem]{Remark} 

\newtheorem{corollary}[theorem]{Corollary}

\DeclareMathOperator*{\supp}{supp}

\newcommand{\noi}{\noindent}

\newcommand{\Z}{\mathbb{Z}}
\newcommand{\R}{\mathbb{R}}
\newcommand{\T}{\mathbb{T}}

\newcommand{\W}{\mathcal{W}}
\newcommand{\Y}{\mathcal{Y}}

\newcommand{\eps}{\varepsilon}

\newcommand{\ft}{\widehat}

\newcommand{\cj}{\overline}
\newcommand{\dx}{\partial_x}

\numberwithin{equation}{section}
\numberwithin{theorem}{section}

\begin{document}

\title
[First and second order approximations]
{\bf First and second order approximations for a
nonlinear wave equation}

\author{Oana Pocovnicu}

\date{\today}

\begin{abstract}
We consider the following nonlinear wave equation:
\begin{equation*}
i\partial_tv-|D|v=|v|^2v,\tag{NLW}
\end{equation*}

\noi
where $D = -i \dx$, both on $\R$ and $\T$.

In the case of $\R$, we prove that if the initial condition
is of order $O(\eps)$ and supported on positive frequencies only,
then the corresponding solution can be approximated by the solution of the 
Szeg\"o equation. The Szeg\"o equation writes $i\partial_tu=\Pi_+(|u|^2u)$,
where $\Pi_+$ is the Szeg\"o projector onto non-negative frequencies,
and is a completely integrable system. The approximation holds for a long time 
$0\leq t\leq \frac{C}{\eps^2}\Big(\log(\frac{1}{\eps^{\delta}})\Big)^{1-2\alpha}$, $0\leq\alpha\leq 1/2$.
The proof is based on the renormalization group method, first introduced
in the context of theoretical physics by Chen, Goldenfeld, and Oono.

As a corollary, we give an example of solution of (NLW) on $\R$ whose
high Sobolev norms inflate over the time,
relatively to the norm of the initial condition.

An analogous result of approximation was proved by 
G\'erard and Grellier \cite{PGSG res}, in the case of $\T$,
using the theory of Birkhoff normal forms.
We improve their result by finding the second order
approximation with the help of an averaging method introduced by
Temam and Wirosoetisno in \cite{Temam Wirosoetisno}. We show that the
effective dynamics will no longer be given by the
Szeg\"o equation.
\end{abstract}

\maketitle

\section{Introduction}
One of the most important properties
in the study of the nonlinear Schr\"odinger equations (NLS) is {\it dispersion}.
It is often exhibited in the form of the Strichartz estimates
of the corresponding linear flow.
In case of the cubic NLS:
\begin{equation}\label{eqn: Schrodinger 4}
i\partial_t u+\Delta u=|u|^2u, \quad (t,x)\in\R\times M,
\end{equation}

\noi
Burq, G\'erard, and Tzvetkov \cite{BGT} observed
that the dispersive properties are strongly influenced
by the geometry of the underlying manifold $M$.
Taking this idea further, G\'erard and Grellier \cite{PGSG} remarked that
dispersion  disappears completely
when $M$ is a sub-Riemannian manifold
or when the Laplacian is replaced by
the Grushin operator. In those cases, by conveniently
decomposing the function $u$,
we obtain that at least in the radial case,
the Schr\"odinger equation is equivalent to
the following system of transport equations:
\begin{align}\label{sys Schrodinger}
i(\partial_t\pm(2m+1)\partial_x)u_m=\Pi_m(|u|^2u),
\end{align}

\noi
where $\Pi_m$ are pseudo-differential orthogonal projectors.
Therefore, studying the Schr\"odinger equation in a non-dispersive situation
comes down to studying a system of the above type.

In this paper we consider the following nonlinear wave equation on $\R$ and $\T$:
\begin{equation}\label{eq:dirac}
\begin{cases}
i\partial_tv-|D|v=|v|^2v,\tag{NLW}\\
v(0)=v_0
\end{cases}
\end{equation}

\noi
where $D = -i \dx$. It is indeed a nonlinear wave equation since
by applying the operator $i\partial_t+|D|$ to both sides of the equation,
we obtain:
\begin{align*}
-\partial_{tt}v+\Delta v=|v|^4v+2|v|^2(|D|v)-v^2(|D|\bar{v})+|D|(|v|^2v).
\end{align*}

\noi
Equation \eqref{eq:dirac} was studied
on $\T$ by G\'erard and Grellier in \cite{PGSG res}.

We consider the Hardy spaces on the unit disc and upper upper-half plane:
\begin{align*}
L^2_+(\T)&=\{f\in L^2(\T); \,\hat{f}(k)=0 \text{ if } k<0\},\\
L^2_+(\R)&=\{f\in L^2(\R); \,\supp\hat{f}\subset [0,\infty)\},
\end{align*}

\noi
and the corresponding Sobolev spaces $H^s_+(\T)=L^2_+(\T)\cap H^s(\T)$
and $H^s_+(\R)=L^2_+(\R)\cap H^s(\R)$, $s\geq 0$.

The Szeg\"o projector onto the Hardy space of the unit disc is $\Pi_+:L^2(\T)\to L^2_+(\T)$, defined by:
\begin{align*}
\Pi_+f(x)=\sum_{k=0}^{\infty}\hat{f}(k)e^{ikx}.
\end{align*}

\noi
In the case of $\R$, the Szeg\"o projector $\Pi_+:L^2(\R)\to L^2_+(\R)$
can be defined similarly by:
\begin{align*}
\ft{\Pi_+f}(\xi)=
\begin{cases}
\hat{f}(\xi), \text{ if } \xi\geq 0,\\
0, \text{ if } \xi< 0
\end{cases}
\end{align*}

\noi
We also define $\Pi_-=I-\Pi_+$, where $I$ is the identity operator.
Applying the projectors $\Pi_+$ and $\Pi_-$
and writing $v=v_++v_-$, where $v_+=\Pi_+(v)$, and $v_-=\Pi_-(v)$,
we obtain that equation \eqref{eq:dirac} is
equivalent to the following system:
\begin{align}\label{system}
\begin{cases}
&i(\partial_t v_++\partial_x v_+)=\Pi_+ (|v|^2v)\\
&i(\partial_t v_--\partial_x v_-)=\Pi_-(|v|^2v).
\end{cases}
\end{align}

\noi
Notice that this is a system of transport equations
 similar to the one obtained from the Schr\"odinger equation \eqref{sys Schrodinger}.
We expect that the study of this system
and therefore the study of the \eqref{eq:dirac} equation help us
understand better NLS in the case of lack of dispersion.

\medskip

The \eqref{eq:dirac} equation
 is a Hamiltonian evolution associated to the Hamiltonian
\begin{equation*}
E(v)=\frac{1}{2}(|D|v,v)+\frac{1}{4}\|v\|_{L^4}^4,
\end{equation*}

\noi
with respect to the symplectic form $\omega(u,v)=\textup{Im}\int u\bar{v}dx$.
From this structure, we obtain the formal conservation law  of energy $E(v(t))=E(v(0))$.
The invariance under translations and under modulations provides two more conservation laws,
$Q(v(t))=Q(v(0))$ and $M(v(t))=M(v(0))$, where
\begin{equation*}
Q(v)=\|v\|_{L^2}^2\quad \text{and} \quad M(v)=(Dv,v).
\end{equation*}

\noi
The conservation of the mass and energy yields a uniform
bound on the $H^{1/2}$-norm of the solution of \eqref{eq:dirac}.
 Therefore it seems natural to
study the well-posedness of \eqref{eq:dirac}
in $H^{1/2}$. The following result from \cite{PGSG res} states that indeed, the \eqref{eq:dirac}
 equation on $\T$ is globally well-posed in
$H^{\frac{1}{2}}(\T)$.
\begin{proposition}[\cite{PGSG res}]
The nonlinear wave equation \eqref{eq:dirac} is globally well-posed in  $H^{\frac{1}{2}}(\T)$.
Moreover, if $v_0\in H^{s}(\T)$ for some $s>\frac{1}{2}$,
then $v\in C(\R;H^{s})$.
\end{proposition}

\noi
An analogous result holds for \eqref{eq:dirac} equation
on $\R$.

\medskip

In this paper we prove that the solution of the \eqref{eq:dirac} equation on $\R$
with an initial condition of order $O(\eps)$ and supported only on
positive frequencies, can be approximated by the solution of a simpler equation with the same
initial data. The approximation is of order $O(\eps^2)$
and holds for a long time. The approximate equation
is the Szeg\"o equation, recently introduced by G\'erard and Grellier:
\begin{equation}\label{Szego simple}
i\partial_t u=\Pi_+ (|u|^2u).
\end{equation}

\noi
This equation was studied in details on $\T$ in \cite{PGSG, PGSGIT} and on $\R$ in \cite{pocov1, pocov2}.
It is globally well-posed in $H_+^s(\T)$ and $H_+^s(\R)$ for $s\geq 1/2$.
Its most remarkable property is that it is completely integrable,
in the sense that it admits a Lax pair. In particular, it possesses an infinite sequence of conservation laws,
the strongest one being the $H^{1/2}_+$-norm.

The approximation result for the \eqref{eq:dirac} equation on $\R$
was motivated by a similar one proved by G\'erard, Grellier
\cite{PGSG res} in the case of $\T$. The case of $\R$ brings new difficulties related, as we see below,
to low frequencies. Moreover, the method used in the case of $\T$ is the theory
of Birkhoff normal forms. It seems difficult to use normal forms on $\R$ due to small divisors problems. Our result will
be proved using the renormalization group method of Chen, Goldenfeld and Oono \cite{CGO 1994, CGO 1996}
coming from theoretical physics.

\medskip

The heuristic idea that motivated our result on $\R$ and the previous result on $\T$ in \cite{PGSG res}
is the following.
Consider the \eqref{eq:dirac} equation  with an initial condition
$v_0$ such that $v_0=\eps u_0$, where $u_0\in H_+^{1/2}$.
Since we have conservation of the
momentum and of the energy, it follows that $2E(v(t))-M(v(t))=2E(v_0)-D(v_0)$. This yields:
\begin{align*}
2\big(|D|v_-(t),v_-(t)\big)+\frac{1}{2}\|v(t)\|^4_{L^4}=\frac{1}{2}\|v_{0,+}\|^4_{L^4}=O(\eps^4).
\end{align*}

\noi
Thus, $\|v_-(t)\|_{\dot{H}^{1/2}(\T)}=O(\eps^2)$ for all $t\in\R$.
Moreover, we have
\begin{align*}
\|v_-(t)\|^2_{H^{1/2}(\T)}=\sum_{k\leq -1}(1+|k|^2)^{1/2}|\hat{v}(k)|^2\leq 2\sum_{k\leq -1}|k|
|\hat{v}(k)|^2\leq 2 \|v_-(t)\|^2_{\dot{H}^{1/2}(\R)}=O(\eps^4).
\end{align*}

\noi
Then, $\|v_-(t)\|_{H^{1/2}(\T)}=O(\eps^2)$.
Therefore, $v_-(t)$ is $\eps^2$-small, while the solution $v(t)$ is only $\eps$-small.
It seems thus that the dynamics
of \eqref{eq:dirac} is dominated by $v_+(t)$. We omit then all the terms containing
$v_-$ in the nonlinearity of the first equation in \eqref{system}, since they are supposed to be small.
We obtain that $u(t,x)=v_+(t,x+t)$ almost satisfies the Szeg\"o equation
\begin{equation*}
i\partial_t u=\Pi_+ (|u|^2u).
\end{equation*}

\noi
Hence, it is natural to expect that the Szeg\"o equation
provides us with an approximation of the \eqref{eq:dirac} equation
with a small initial condition supported
on positive frequencies.

In the case of $\R$, the conservation
of energy and momentum still gives $\|v_-(t)\|_{\dot{H}^{1/2}(\R)}=O(\eps^2)$, while we have that $\|v(t)\|_{H^{1/2}(\R)}=O(\eps)$
for all $t\in\R$. However, we have no other information on the $L^2$-norm
of $v_-(t)$. This suggests that the low frequencies cause some new difficulty in proving that $v_-$ is small,
and thus in proving that the flow of \eqref{eq:dirac} can be approximated by that of the Szeg\"o equation.

\medskip

In what follows we state a weaker version of the approximation result for the \eqref{eq:dirac} equation on $\T$ in
\cite{PGSG res}. The original result holds for a slightly longer time 
$0\leq t\leq \frac{1}{\eps^2}\log\big(\frac{1}{\eps^{\delta}}\big)$
and without assuming any bound on the solution of the Szeg\"o equation. 
However, in the proof the authors use the complete integrability of 
the Szeg\"o equation, while in Section 3, we will prove this weaker 
version without using the complete integrability.
\begin{theorem}[G\'erard-Grellier \cite{PGSG res}]\label{Thm T}
Let $0<\eps\ll 1$, $0\leq \alpha\leq 1/2$, and $\delta>0$
sufficiently small. Let $s>\frac{1}{2}$ and $W_0\in H^s_+(\T)$.
Let $v(t)$ be the solution of the \eqref{eq:dirac} equation on $\T$
\begin{equation}
\begin{cases}
i\partial_tv-|D|v=|v|^2v\\
v(0)=\W_0:=\eps W_0.
\end{cases}
\end{equation}

\noi
Denote by $\W \in C(\R, H^{1/2}_+(\T))$
the solution of the Szeg\"o equation on $\T$:
\begin{equation}
\begin{cases}
i\partial_t\mathcal{W}=\Pi_+(|\mathcal{W}|^2\mathcal{W})\\
\mathcal{W}(0)=\W_0
\end{cases}
\end{equation}

\noi
with the same initial data.
Suppose that $\|\W(t)\|_{H^s}\leq C\eps\Big(\log(\frac{1}{\eps^{\delta}})\Big)^{\alpha}$
for all $t\in\R$.

Then, if $0\leq t\leq \frac{1}{\eps^2}\Big(\log(\frac{1}{\eps^{\delta}})\Big)^{1-2\alpha}$, we have
\begin{equation*}
\|v(t)-e^{-i|D|t}\mathcal{W}(t)\|_{H^s}\leq \eps^{3-C_0\delta},
\end{equation*}

\noi
where $C_0>0$ is an absolute constant.
\end{theorem}

In the second half of this paper we improve the above result on $\T$.
We find a second order approximate solution,
given by an equation which is more complex
 than the Szeg\"o equation, but which provides a smaller error
 of order
$\eps^5$ instead of $\eps^3$,
in the approximation.
For this purpose, we use the averaging method introduced
by Temam and Wirosoetisno in  \cite{Temam Wirosoetisno}.

In what follows we state and briefly comment the main results of the paper.

\subsection{Main results}

First, in the case of $\R$, we consider an initial condition for \eqref{eq:dirac}
which is supported on positive frequencies only, is of order $O(\eps)$,
and such that the corresponding solution of the Szeg\"o equation is bounded
for all times by $C\eps\Big(\log(\frac{1}{\eps^{\delta}})\Big)^{\alpha}$, $0\leq\alpha\leq 1/2$.
Then the solution of the \eqref{eq:dirac} equation with this initial condition stays $\eps^2$-close to the solution of the Szeg\"o
equation with the same initial condition,
for times $0\leq t\leq \frac{C}{\eps^2}\Big(\log(\frac{1}{\eps^{\delta}})\Big)^{1-2\alpha}$.

\begin{theorem}\label{Main theorem}
Let $0<\eps\ll 1$, $s> \frac{1}{2}$, and $W_0\in H^s_+(\R)$.
Let $v(t)$ be the solution of the \eqref{eq:dirac} on $\R$
\begin{equation}\label{NLW main thm}
\begin{cases}
i\partial_tv-|D|v=|v|^2v\\
v(0)=\W_0=\eps W_0.
\end{cases}
\end{equation}

\noi
Denote by
$\mathcal{W}\in C(\R, H^{s}_+(\R))$ the solution of the Szeg\"o equation on $\R$
\begin{equation}\label{eqn mathcal W}
\begin{cases}
i\partial_t\mathcal{W}=\Pi_+(|\mathcal{W}|^2\mathcal{W})\\
\mathcal{W}(0)=\W_0
\end{cases}
\end{equation}

\noi
with the same initial data.
Assume that there exist $0\leq\alpha\leq \frac{1}{2}$ and
$\delta>0$ small enough such that $\|\W(t)\|_{H^s}\leq C\eps \Big(\log(\frac{1}{\eps^\delta})\Big)^{\alpha}$
for all $t\in\R$.

Then, if $0\leq t\leq \frac{1}{\eps^2}\Big(\log(\frac{1}{\eps^{\delta}})\Big)^{1-2\alpha}$, we have that
\begin{equation*}
\|v(t)-e^{-i|D|t}\mathcal{W}(t)\|_{H^s}\leq C_{\ast}\eps^{2-C_0\delta},
\end{equation*}

\noi
where $C_0>0$ is an absolute constant and $C_{\ast}$ is a constant depending only on
the $H^{1/2}_+(\R)$-norm of $W_0$.
\end{theorem}

Notice that the approximation in Theorem \ref{Thm T} in \cite{PGSG res} for the case of $\T$,
 is better than the one in Theorem
\ref{Main theorem} for the case of $\R$
($\eps^3$ instead of $\eps^2$).
 This is what we expected even from our
heuristic argument above. We will see in the proof that the estimates we
have in the case of $\R$ are worse than those for the case of $\T$,
due to low frequencies.

\begin{remark}
\rm
Theorem \ref{Thm T} was proved in \cite{PGSG res} using the theory of Birkhoff
normal forms. This method seems to be difficult to adapt to the case of $\R$.
The method we use in this paper is the renormalization group method,
coming from theoretical physics. The two methods are intimately related. In \cite{Ziane},
it was noticed that, for a large class of autonomous ODEs, the nonlinearity which appears in the
RG equation of order one is actually the Birkhoff normal form. This result was extended in
\cite{De Ville} to order two, for the same class of autonomous ODEs, and to first order,
for a class of non-autonomous ODEs. The advantage of the RG method over the normal form theory
is that the secular terms are more readily identified by inspection of a naive perturbation expansion,
than by inspection of the vector field.
\end{remark}

\medskip
The purpose of the approximation Theorem \ref{Main theorem}
is to deduce some information on the \eqref{eq:dirac}
equation from the known results one has for
the Szeg\"o equation. Some particularly interesting
solutions of the Szeg\"o equation are those whose initial conditions
are non-generic rational functions, for example $W_0=\frac{1}{x+i}-\frac{2}{x+2i}$.
For such solutions, we proved in \cite{pocov2} the following result:

\begin{proposition}[\cite{pocov2}]\label{prop Szego}
Let $s> 1/2$. Let $W\in C(\R,H^{s}_+(\R))$ be
the solution of the Szeg\"o equation
\[i\partial_tW=\Pi_+(|W|^2W)\]

\noi
with non-generic initial condition $W_0=\frac{1}{x+i}-\frac{2}{x+2i}\in H^s_+(\R)$. Then,
for $t$ large enough, there exist $C,c>0$ such that
\[c\,t^{2s-1}\leq \|W(t)\|_{H^s}\leq Ct^{2s-1}.\]

\noi
In particular, $\|W(t)\|_{H^s}\to\infty$ as $t\to\infty$.
\end{proposition}

The following corollary proves that the high Sobolev norms of the \eqref{eq:dirac}
equation with initial condition $\eps W_0=\frac{\eps}{x+i}-\frac{2\eps}{x+2i}$
grow relatively to the norm of the initial condition.
\begin{corollary}\label{Cor}
Let $0<\eps\ll 1$, $s>\frac{1}{2}$, and $\delta>0$ sufficiently small.
 Let $W_0\in H^s_+(\R)$ be the non-generic rational function
$W_0=\frac{1}{x+i}-\frac{2}{x+2i}$.
Denote by $v(t)$ be the solution of the \eqref{eq:dirac} equation on $\R$
\begin{equation*}
\begin{cases}
i\partial_tv-|D|v=|v|^2v\\
v(0)=\eps W_0.
\end{cases}
\end{equation*}

\noi
Then, for $\frac{1}{2\eps^2}\Big(\log(\frac{1}{\eps^{\delta}})\Big)^{\frac{1}{4s-1}}\leq t\leq \frac{1}{\eps^2}\Big(\log(\frac{1}{\eps^{\delta}})\Big)^{\frac{1}{4s-1}}$,
we have that
\begin{equation*}
\frac{\|v(t)\|_{H^s(\R)}}{\|v(0)\|_{H^s(\R)}}\geq C\Big(\log(\frac{1}{\eps^{\delta}})\Big)^{\frac{4s-2}{4s-1}}.
\end{equation*}
\end{corollary}

A similar result is available for the case of $\T$ \cite{PGSG res}.

The time on which the approximation
in Theorem \ref{Main theorem} is available, 
$t\leq \frac{1}{\eps^2}\big(\log(\frac{1}{\eps^{\delta}})\Big)^{1-2\alpha}$,
does not allow us to prove the existence of a time $t^{\eps}$
such that $\|v(t^{\eps})\|_{H^s(\R)}\to\infty$ as $\eps\to 0$.
For that to happen, we would need an approximation at least up to a time of order
$\frac{1}{\eps^{2+\beta}}$ where $\beta>0$.

\medskip

In the case of $\T$, we find the
second order approximation, that is an approximation
with an error of order $\eps^5$ instead of $\eps^3$. We notice that the
effective dynamics are no longer given by the Szeg\"o equation.

\begin{theorem}\label{Theorem second interation}
Let $0<\eps\ll 1$, $s>\frac{1}{2}$, $0\leq\alpha\leq\frac{1}{2}$, and
$\delta>0$ small enough. Let $W_0\in H^s_+(\T)$ be such that
the solution of the Szeg\"o equation \eqref{Szego simple} with initial
condition $\eps W_0$ is uniformly bounded by $\eps\Big(\log(\frac{1}{\eps^{\delta}})\Big)^{\alpha}$
for all $t\in\R$.
Denote by $v(t)$ the solution of the \eqref{eq:dirac} equation on $\T$
\begin{equation*}
\begin{cases}
i\partial_tv-|D|v=|v|^2v\\
v(0)=\W_0=\eps W_0.
\end{cases}
\end{equation*}

\noi
Consider
$\mathcal{W}\in C(\R, H^{s}_+(\T))$ to be the solution of the
following equation on $\T$:
\begin{equation}\label{eqn: second order}
\begin{cases}
i\partial_t\mathcal{W}=\Pi(|\mathcal{W}|^2\mathcal{W})-\Pi_+(|\W|^2\frac{1}{D}
\Pi_-(|\W|^2\W))-\frac{1}{2}\Pi_+(\W^2\frac{1}{D}\cj{\Pi_-(|\W|^2\W)})\\
\mathcal{W}(0)=\W_0.
\end{cases}
\end{equation}

\noi
with the same initial condition.

For a function $h\in H^s(\T)$, set
\[f_{\textup{osc}}(h,t)=e^{i|D|t}(|e^{-i|D|t}h|^2e^{-i|D|t}h)
-\frac{1}{2\pi}\int_0^{2\pi} e^{i|D|\tau}(|e^{-i|D|\tau}h|^2e^{-i|D|\tau }h)d\tau.\]

\noi
Denote by $F_{\textup{osc}}(h,t)$ the unique function
of mean zero in $t$ such that $\frac{\partial{F_{\textup{osc}}}}{\partial t}(h,t)=f_{\textup{osc}}(h,t)$.
Consider
\[v_{\textup{app}}(t)=e^{-i|D|t}\big(\mathcal{W}(t)+F_{\textup{osc}}(\mathcal{W}(t),t)\big).\]

\noi
Then, if $0\leq t\leq \frac{1}{\eps^2}\Big(\log(\frac{1}{\eps^{\delta}})\Big)^{1-2\alpha}$, we have
\begin{equation*}
\|v(t)-v_{\textup{app}}(t)\|_{H^s}\leq \eps^{5-C_0\delta},
\end{equation*}

\noi
where $C_0>0$ is an absolute constant.
\end{theorem}

The above result cannot be directly extended to the case of $\R$.
The main reason is that in equation \eqref{eqn: second order}
we see appear the operator $\frac{1}{D}\Pi_-$. In the case of $\T$,
we have that $\frac{1}{D}\Pi_-e^{ikx}=\frac{1}{k}\pmb{1}_{k\leq -1}$
and thus there is no problem related to small divisors.
However, in the case of $\R$, if we pass into the Fourier space,
we have $\frac{1}{\xi}\pmb{1}_{\xi<0}$ and when $\xi$
approaches zero, this gives a singularity.  A way to get around this singularity
would be to consider instead of resonances, i.e. frequencies for which a certain phase is null $\phi=0$,
almost resonances $|\phi|\leq \gamma$, for an optimal $\gamma>0$.
However, it seems that this would complicate significantly the dynamics \eqref{eqn: second order}.

In order to prove Theorem \ref{Theorem second interation}, we use an averaging method
introduced by Temam and Wirosoetisno in \cite{Temam Wirosoetisno}.

We briefly describe in what follows the renormalization method, the averaging method,
the concept of resonance,
and their usage in the literature.

\subsection{The renormalization group method, the averaging \,\,method, and the concept of resonance}

The renormalization group (RG) method was introduced by Chen, Goldenfeld, and Oono
\cite{CGO 1994, CGO 1996} in the context of theoretical physics,
as a unified tool for asymptotic analysis. Its origin goes back to perturbative quantum field theory.

The method is most often used to find a long-time approximate solution to a
perturbed equation. The main advantage of
the RG method is that it provides an algorithm that can be easily applied to
 many equations. The starting point is a
naive perturbation expansion, so that one does not need to guess or to make
ad hoc assumptions about the structure of the perturbation series.
Then, the divergent terms in the expansion(unbounded in time), are removed by renormalization.
This leads to introducing the renormalization group equation.
The solution of the RG equation is the main part of an
approximate solution.

The effectiveness of the RG method was illustrated in a variety of examples of ordinary differential equations
traditionally analyzed using disparate methods, including the method of multiple scales, boundary layer theory,
the WKBJ method, the Poincar\'e-Lindstedt method, and the method of averaging.

The method was justified mathematically for a large class of ODEs in \cite{Ziane, De Ville}. It was also
rigorously applied to some PDEs on bounded intervals,
namely the Navier-Stokes equations \cite{Moise Temam}, a slightly compressible fluid equation and the
Swift-Hohenberg equation \cite{Moise Ziane},
and the primitive equations of the atmosphere and the ocean \cite{Petcu Temam}. In \cite{Abou Salem}
it was applied to the quadratic nonlinear Schr\"odinger equation on $\R^3$.

The idea behind the RG method is that the dynamics of an equation is dominated by its resonant part.
This idea is also used by
Colliander, Keel, Staffilani, Takaoka, and Tao in \cite{I team}
to prove the existence of solutions for the cubic non-linear Schr\"odinger equation on $\T^2$ with
arbitrarily large high Sobolev norms. They consider a reduced resonant equation for which they prove growth of high
Sobolev norms, and then show that this resonant equation provides a good approximation for the initial one.

\medskip

The averaging method we use in this paper was introduced by Temam and Wirosoetisno in \cite{Temam Wirosoetisno}
in the context of a class of differential equations. At first order it is related to the RG method,
while at higher orders it is related to the asymptotic expansions of
Bogolyubov and Mitropol'skii \cite{Bogolyubov and Mitropol'skii}.

The RG method can also be applied at higher orders, as it was done for ODEs in \cite {De Ville}.
In the case of the \eqref{eq:dirac} equation on $\T$, we could prove that at second order the RG equation
is exactly the averaged equation \eqref{eqn: second order} in
Theorem \ref{Theorem second interation}. However, the computations
one needs to do
when applying the RG method at second order
are much more tedious
than when applying the averaging method.
 Another reason why we preferred to present the averaging method
for the second order approximation,
is that this method does not only give the effective dynamics \eqref{eqn: second order},
but also gives an algorithm of how to build an approximate solution
and how to estimate the error, which is not clear when one applies the RG method at higher orders.

\medskip

Both the RG and the averaging methods are based on the concept of decomposing the nonlinearity into its resonant and
non-resonant parts. Such a decomposition was very effective in proving global existence of small solutions
of dispersive equations and scattering. This was done in several works of Germain, Masmoudi, and Shatah
\cite{Germain 1, Germain 2, Germain 3, Germain 4, Germain 5, Shatah}, who treated the case of the
gravity water waves equation in dimension 3, the coupled Klein-Gordon equations with different speeds,
 and the quadratic nonlinear Schr\"odinger equation in dimension 2 and 3. Gustafson, Nakanishi, and Tsai
treated the case of the Gross-Pitaevskii equation in dimension 3  in \cite{GNT}.
They use time, space, and
space-time resonances, whereas in this
paper we only consider time resonances.

The specificity of the \eqref{eq:dirac} equation is that {\it the resonant set does not have measure
zero}, as it was the case in the above cited papers.
For this reason it is natural not to expect scattering, but a long-time approximation of
the solution by some effective dynamics governed
by the effect of the resonant part of the non-linearity.
The decomposition in resonant and non-resonant part, was used in \cite{Abou Salem},
precisely in this purpose in the case of the quadratic Schr\"odinger equation
in dimension 3.

\medskip

The structure of the paper is as follows. In the rest of the introduction,
we heuristically explain the need of splitting the nonlinearity into its
resonant and  oscillatory part, which is at the basis of both the RG and
averaging method. In Section 2, we present the RG method and use it to prove
Theorem \ref{Main theorem}, dealing with the first order approximation in the
case of $\R$. We also prove Corollary \ref{Cor}
which refers to high Sobolev norm inflation in the case of
non-generic initial data. To have a good comparison between the case of $\R$
and that of $\T$, and for a better understanding
of the second order approximation in the case of $\T$,
in Section 3 we re-prove Theorem \ref{Thm T}
from \cite{PGSG res} using the RG method.  In Section
4, we present the averaging method at second order and use it to prove
Theorem \ref{Theorem second interation} treating the second order approximation
in the case of $\T$.

\subsection{Heuristics of the proof of Theorem \ref{Main theorem}}

The first approach to proving Theorem \ref{Main theorem} is the following one.
Consider the change of variables $u(t)=\frac{1}{\eps}e^{i|D|t}v(t)$. Then $u$ satisfies the equation:
\begin{equation}\label{eqn u}
\begin{cases}
\partial_tu=-i\eps^2 e^{i|D|t}(|e^{-i|D|t}u|^2e^{-i|D|t}u)\\
u(0)=W_0.
\end{cases}
\end{equation}

\noi
Let us now set $W(t):=\frac{\mathcal{W}(t)}{\eps}$. Then $W(t)$ satisfies
\begin{equation}\label{eq:W}
\begin{cases}
i\partial_tW=\eps^2\Pi_+(|W|^2W)\\
W(0)=W_0.
\end{cases}
\end{equation}

\noi
Then, setting $w(t)=u(t)-W(t)$, we have
\[\|v(t)-e^{-i|D|t}\mathcal{W}\|_{H^s}=\eps\|e^{-i|D|t}\big(u(t)-W(t)\big)\|_{H^s}=\eps\|w(t)\|_{H^s}.\]
We have that $w$ satisfies the equation
\begin{equation*}
\begin{cases}
\partial_tw=-i\eps^2 e^{i|D|t}(|e^{-i|D|t}u|^2e^{-i|D|t}u)+i\eps^2\Pi_+(|W|^2W)\\
w(0)=0.
\end{cases}
\end{equation*}

\noi
Therefore,
\begin{align*}
w(t)=-i\eps^2 \int_0^t\Big(e^{i|D|\tau}(|e^{-i|D|\tau}u|^2e^{-i|D|\tau}u)-\Pi_+(|W(\tau)|^2W(\tau))\Big)d\tau
\end{align*}

\noi
The classical technique of estimating $w(t)$ consists in writing the right-hand side in such a way that we see $w(\tau)$ appear
under the integral, and then
use Gronwall's inequality. However, $w(\tau)=u(\tau)-W(\tau)$, and in the above relation the only term in which $u$
appears is $f(u,\tau):=-ie^{i|D|\tau}(|e^{-i|D|\tau}u|^2e^{-i|D|\tau}u)$. It is thus natural
to decompose the term $f(u,\tau)$ into a part which does not explicitly
depend on $\tau$ called the resonant part, $f_{\textup{res}}(u)$,
and a part which depends on $\tau$ called the oscillatory part, $f_{\textup{osc}}(u,\tau)$.
Then, $f_{\textup{res}}(u)-\Pi_+(|W|^2W)$ provides us with a term $w=u-W$.

Since we have more information on $W(\tau)$, which can be transformed with
a simple change of variables into the solution of the Szeg\"o equation \eqref{Szego simple}, it may be more convenient
to decompose $f(W,\tau)=-ie^{i|D|\tau}(|e^{-i|D|\tau}W|^2e^{-i|D|\tau}W)$. It turns out that its resonant part is exactly
$-i\Pi_+(|W|^2W)$ and thus
\[f(W(\tau),\tau)=-i\Pi_+(|W(\tau)|^2W(\tau))+f_{\textup{osc}}(W(\tau),\tau).\]
Therefore,
\begin{align*}
w(t)=\eps^2 \int_0^t\Big(f(u(\tau),\tau)-f(W(\tau),\tau)\Big)d\tau+\int_0^tf_{\textup{osc}}(W(\tau),\tau)d\tau
\end{align*}

\noi
The first term will indeed yield $w=u-W$, and we are left with estimating the integral of the oscillatory part
$f_{\textup{osc}}(W(\tau),\tau)$. Since it depends on $\tau$ both explicitly and implicitly, it turns out that it can be difficult to
estimate its integral. For that reason we consider in the following $F_{\textup{osc}}(W(t),t)=\int_0^tf_{\textup{osc}}(W(t),\tau)d\tau$, where
the integrand depends only explicitly on $\tau$. We construct an ansatz using $F_{\textup{osc}}(W,t)$ and we prove that with this ansatz,
the error is indeed small.

\section{First order approximation for the \eqref{eq:dirac} \,\,equation on $\R$}

\subsection{The renormalization group method at order one}

In what follows we describe the RG method of first order in the case of the \eqref{eq:dirac} equation on $\R$.

In the \eqref{eq:dirac} equation, we make the change of variables $u(t)=\frac{1}{\eps}e^{i|D|t}v(t)$
and set $\tilde{\eps}:=\eps^2$. Then $u$ satisfies the equation:
\begin{equation}\label{eqn u 0}
\begin{cases}
\partial_tu=-i\tilde{\eps} e^{i|D|t}(|e^{-i|D|t}u|^2e^{-i|D|t}u)=:\tilde{\eps} f(u,t)\\
u(0)=\frac{1}{\eps}v_0=:u_0.
\end{cases}
\end{equation}

The starting point of the RG method is the naive perturbation expansion
\begin{equation*}
u(t)=u^{(0)}(t)+\tilde{\eps} u^{(1)}(t)+\tilde{\eps}^2 u^{(2)}(t)+\dots
\end{equation*}

\noi
Taylor-expanding $f(u,t)$ around $u^{(0)}$, we obtain
\begin{equation*}
f(u,t)=f(u^{(0)},t)+f'(u^{(0)},t)(u(t)-u^{(0)}(t))+\dots =f(u^{(0)},t)+\tilde{\eps} f'(u^{(0)},t)u^{(1)}(t)+\dots
\end{equation*}

\noi
Plugging the last two expansions into the equation \eqref{eqn u 0} and identifying the coefficients according to the
powers of $\tilde{\eps}$, we obtain:
\begin{equation}\label{eqn u 2}
\begin{cases}
\partial_tu^{(0)}=0\\
\partial_tu^{(1)}=f(u^{(0)},t)\\
\partial_tu^{(2)}=f'(u^{(0)},t)\cdot u^{(1)}(t)\\
\dots
\end{cases}
\end{equation}

\noi
Therefore, $u^{(0)}(t)=u_0$ for all $t\in\R$, and using Duhamel's formula we have
\begin{equation*}
u^{(1)}(t)=\int_0^tf(u_0,s)ds.
\end{equation*}

\noi
Here we assumed that $u^{(1)}(0)=0$. As it was shown in \cite{Ziane},
this assumption does not cause a loss of generality for an approximation of order $\tilde{\eps}$.
Thus, if we look for an approximation of the solution up to order
$O(\tilde{\eps})$ and neglect any terms $O(\tilde{\eps}^2)$,
we have
\begin{equation}\label{eqn u 1}
u(t)=u_0+\tilde{\eps} u^{(1)}(t)+O(\tilde{\eps}^2)=u_0+\tilde{\eps}\int_0^tf(u_0,s)ds+O(\tilde{\eps}^2).
\end{equation}

Now we decompose the nonlinearity $f(u,t)$ into its resonant and non-resonant part.
In order to do that, we first write the nonlinearity in the Fourier space:
\begin{align*}
\mathcal{F}\big(f(u,s)\big)&(\xi)=-ie^{i|\xi|s}\mathcal{F}\big(|e^{-i|D|s}u|^2e^{-i|D|s}u\big)(\xi )\\
=&-ie^{i|\xi|s}\int_{\R}\mathcal{F}\big((e^{-i|D|s}u)^2\big)(\eta)\mathcal{F}\big(\cj{e^{-i|D|s}u}\big)(\xi-\eta )d\eta\\
=&-ie^{i|\xi|s}\int_{\R}\int_{\R}\mathcal{F}\big(e^{-i|D|s}u\big)(\eta-\zeta)
\mathcal{F}\big(e^{-i|D|s}u\big)(\zeta)\mathcal{F}\big(\cj{e^{-i|D|s}u})\big)(\xi-\eta )d\zeta d\eta\\
=&-i\int_{\R}\int_{\R}e^{is(|\xi|-|\zeta|+|\eta-\xi|-|\eta-\zeta|)}
\hat{u}(\eta-\zeta)\hat{u}(\zeta)\cj{\hat{u}}(\eta-\xi )d\zeta d\eta.
\end{align*}

\noi
Setting $\phi (\xi,\eta,\zeta):=|\xi|-|\zeta|+|\eta-\xi|-|\eta-\zeta|$, we can write
\begin{equation}\label{decomp f}
f(u,s)=f_{\textup{res}}(u)+f_{\textup{osc}}(u,s),
\end{equation}

\noi
where
\begin{align}\label{f_res, f_osc}
f_{\textup{res}}(u)&=-i\mathcal{F}^{-1}\int\int_{\{\phi=0\}}
\hat{u}(\eta-\zeta)\hat{u}(\zeta)\cj{\hat{u}}(\eta-\xi )d\zeta d\eta,\\
f_{\textup{osc}}(u,s)&=-i\mathcal{F}^{-1}\int\int_{\{\phi\neq 0\}}
e^{is(|\xi|-|\zeta|+|\eta-\xi|-|\eta-\zeta|)}\hat{u}(\eta-\zeta)\hat{u}(\zeta)\cj{\hat{u}}(\eta-\xi )d\zeta d\eta.\notag
\end{align}

\noi
As it will be proved in Lemma \ref{phase 0 R} in the next section, for fixed $\xi$, the set
$\{\phi(\xi,\eta,\zeta)=0\}\subset\R^2$
has non-zero Lebesgue measure, and thus it makes sense to integrate on this set. More precisely,
$\{\phi(\xi,\eta,\zeta)=0\}$ is the set of $(\eta,\zeta)\in\R^2$ such that $\zeta, \eta-\xi, \eta-\zeta$ have the same sign as $\xi$,
or $\zeta=\xi$, or $\eta-\zeta=\xi$.

Plugging the decomposition \eqref{decomp f} into the equation \eqref{eqn u 1}, we obtain
\begin{equation*}
u(t)=u_0+\tilde{\eps} tf_{\textup{res}}(u_0)+\tilde{\eps}\int_0^tf_{\textup{osc}}(u_0,s)ds+O(\tilde{\eps}^2).
\end{equation*}

\noi
We notice that the resonant part of the non-linearity,
which is constant in time, causes the appearance of the secular term $\tilde{\eps} tf_{\textup{res}}(u_0)$.
This term will grow with time and will cause the approximation to break down as time approaches
$\frac{1}{\tilde{\eps}}$. The purpose of the renormalization group method consists in re-normalizing the secular term.
By doing that, its main contribution is taken into account in such a way that the approximation of $u$
stays valid at least up to a time of order $\frac{1}{\tilde{\eps}}$.
The idea behind the renormalization group method is to regard
the term $u_0+\tilde{\eps} tf_{\textup{res}}(u_0)$ as being the Taylor
expansion of order one of a function $W(t)$ around $t=0$.
Then, one introduces the renormalization group equation:
\begin{equation}\label{eqn W}
\begin{cases}
\partial_tW=\tilde{\eps} f_{\textup{res}}(W)\\
W(0)=u_0
\end{cases}
\end{equation}

\noi
An approximation of order $O(\tilde{\eps})$ of $u(t)$ is then
\begin{equation*}
u(t)=W(t)+\tilde{\eps} F_{\textup{osc}}(W(t),t),
\end{equation*}

\noi
where we set $F_{\textup{osc}}(h,t):=\int_0^tf_{\textup{osc}}(h,s)ds$ for all $h\in H^{\frac{1}{2}}_+$.

\subsection{Approximate solution for the \eqref{eq:dirac} equation on $\R$}

In this section we construct an approximate solution
based on the solution of the RG equation.
We first determine the resonant part of the
non-linearity $f_{\textup{res}}$. For that purpose we fix $\xi\in\R$,
and determine the area in the $(\zeta,\eta)$-plane
in which $\phi(\xi,\eta,\zeta)$ vanishes.

Let us first make the following notations:
\begin{align*}
\xi_1=\xi, \,\,\,\,\, \xi_2=\zeta,\,\,\,\,\,\xi_3=\eta-\xi,\,\,\,\,\,\xi_4=\eta-\zeta.
\end{align*}

\noi
Notice that $\xi_1-\xi_2+\xi_3-\xi_4=0$. Then, $\phi(\xi,\eta,\zeta)=0$
is equivalent to \\ $|\xi_1|-|\xi_2|+|\xi_3|-|\xi_4|=0$. We have the following
lemma, whose proof follows its analogue in the case of $\T$ \cite{PGSG res}.

\begin{lemma}\label{phase 0 R}
The set of $(\xi_1,\xi_2,\xi_3,\xi_4)\in\R^4$ such that $\xi_1-\xi_2+\xi_3-\xi_4=0$ and
 $|\xi_1|-|\xi_2|+|\xi_3|-|\xi_4|=0$,  is
\begin{align*}
M:=&\{(\xi_1,\xi_2,\xi_3,\xi_4)\in\R^4, \xi_1\neq \xi_2,  \xi_1\neq \xi_4\Big| \xi_1,\xi_2,\xi_3,\xi_4\geq 0\}\\
&\cup \{(\xi_1,\xi_2,\xi_3,\xi_4)\in\R^4, \xi_1\neq \xi_2,  \xi_1\neq \xi_4\Big| \xi_1,\xi_2,\xi_3,\xi_4\leq 0\}\\
&\cup \{(\xi_1,\xi_1,\xi_3,\xi_3)\in\R^4\}\cup \{(\xi_1,\xi_2,\xi_2,\xi_1)\in\R^4\}.
\end{align*}
\end{lemma}

Coming back to the notations in $\xi$, $\eta$, and $\zeta$, we have that $\phi(\xi,\eta,\zeta)=0$ in the following cases:

If $\xi>0$ and $(\eta,\zeta)\in\{(\eta,\zeta)\in\R^2| \eta\geq \xi, \eta\geq \zeta\geq 0\}\cup\{\zeta=\xi\}\cup\{\eta-\zeta=\xi\}$

If $\xi<0$ and $(\eta,\zeta)\in\{(\eta,\zeta)\in\R^2| \eta\leq \xi, \eta\leq \zeta\leq 0\}\cup\{\zeta=\xi\}\cup\{\eta-\zeta=\xi\}$
\medskip

\noi
Since, for fixed $\xi\in\R$, the sets $\{\zeta=\xi\}$ and
 $\{\eta-\zeta=\xi\}$ are of measure zero in the $(\zeta,\eta)$-plane,
they do not interfere in the integration in equation \eqref{f_res, f_osc},
 and thus we can neglect them. We are therefore left with the following
two terms of $\mathcal{F}\big(f_{\textup{res}}(u)\big)$:

\noi
1. The case $\xi>0$, $\zeta\geq 0$, $\eta-\zeta\geq 0$, $\eta-\xi\geq 0$:
\begin{align*}
&-i\pmb{1}_{\xi> 0}\int\int\hat{u}(\eta-\zeta)\hat{u}(\zeta)\cj{\hat{u}}(\eta-\xi )
\pmb{1}_{\zeta\geq 0}\pmb{1}_{\eta-\zeta\geq 0}\pmb{1}_{\eta-\xi\geq 0}d\zeta d\eta\\
&=-i\pmb{1}_{\xi> 0}\int\int\hat{u}_+(\eta-\zeta)\hat{u}_+(\zeta)\cj{\hat{u}}_+(\eta-\xi )d\zeta d\eta
=-i\mathcal{F}\big(\Pi_+ (|u_+|^2u_+)\big)(\xi)\pmb{1}_{\xi> 0}.\\
\end{align*}

\noi
2. The case $\xi<0$, $\zeta<0$, $\eta-\zeta<0$, $\eta-\xi<0$:
\begin{align*}
&-i\pmb{1}_{\xi< 0}\int\int\hat{u}(\eta-\zeta)\hat{u}(\zeta)\cj{\hat{u}}(\eta-\xi )\pmb{1}_{\zeta< 0}
\pmb{1}_{\eta-\zeta< 0}\pmb{1}_{\eta-\xi< 0}d\zeta d\eta\\
&=-\pmb{1}_{\xi< 0}i\int\int\hat{u}_-(\eta-\zeta)\hat{u}_-(\zeta)\cj{\hat{u}}_-(\eta-\xi )d\zeta d\eta
=-i\mathcal{F}\big(\Pi_- (|u_-|^2u_-)\big)(\xi)\pmb{1}_{\xi< 0}.\\
\end{align*}

\noi
Thus, the resonant part of the nonlinearity is
\begin{align}\label{fres R}
f_{\textup{res}}(u)=-i\Big(\Pi_+(|u_+|^2u_+)+\Pi_-(|u_-|^2u_-)\Big)
\end{align}

\noi
Let $W_0\in H^s_+(\R)$, $s>1/2$.
We consider the renormalization group equation:
\begin{align}\label{eq:RG}
\begin{cases}
\partial_tW=\eps^2f_{\textup{res}}(W)\\
W(0)=W_0
\end{cases}
\end{align}

\noi
Projecting onto non-negative and negative frequencies, we obtain two equations,
 one for $W_+:=\Pi_+(W)$ and one for $W_-:=\Pi_-(W)$. Notice first that, since $W_0\in H^s_+(\R)$,
 we have that $W_{0,-}=0$ and $W_{0,+}=W_0$. Then, the equations we obtain are:
\begin{align*}
\begin{cases}
i\partial_tW_+=\eps ^2\Pi_+ (|W_+|^2W_+) \\
W_+(0)=W_0
\end{cases}
\end{align*}

\noi
and
\begin{align*}
\begin{cases}
i\partial _t W_-=\eps ^2\Pi_-(|W_-|^2W_-)\\
W_-(0)=0.
\end{cases}
\end{align*}

\noi
By the Cauchy-Lipschitz theorem, we have that
 $W_-(t)=0$ for all $t\in\R$, and thus $W=W_+$.
We construct an approximate solution by
\begin{align}\label{def uapp}
u_{\textup{app}}(t)=W(t)+\eps^2F_{\textup{osc}}(W(t),t).
\end{align}

\noi
Then, $u_{\textup{app}}$ satisfies the equation
\begin{align}\label{estimate u_app}
\begin{cases}
\partial_tu_{\textup{app}}=\eps^2f(W(t),t)+\eps^4D_{W}F_{\textup{osc}}(W(t),t)\cdot f_{\textup{res}}(W(t))\\
u_{\textup{app}}=W_0.
\end{cases}
\end{align}

\noi
By the Duhamel formula, we obtain that
\begin{align}\label{Duhamel u_app}
u_{\textup{app}}(t)=W_0+\eps^2\int_0^t f(u_{\textup{app}}(s))ds
+\int_0^t R_{\eps}(s,W(s))ds,
\end{align}

\noi
where
\[R_{\eps}(W(t),t)=\eps^2\Big(f(W(t))-f(u_{\textup{app}}(t))\Big)+\eps^4D_{W}
F_{\textup{osc}}(W(t),t)\cdot f_{\textup{res}}(W(t)).\]

\subsection{Estimates for the oscillatory part of the nonlinearity in the case of $\R$}

\begin{lemma}\label{Est 1 R}
Let $s\geq 1$. Let $W\in C(\R,H^s_+(\R))$ be such that $\W=\eps W$ is the solution of
the Szeg\"o equation \eqref{eqn mathcal W} with initial data $\W_0=\eps W_0$. Then, we have that
\begin{align*}
&\|F_{\textup{osc}}(W,t)\|_{H^s}\leq C_{\ast}t^{1/2}+C\|W\|_{H^s}^3,\\
&\|D_{W}F_{\textup{osc}}(W(t),t)\cdot f_{\textup{res}}(W(t))\|_{H^s}\leq C_{\ast}t^{1/2}+C\|W\|^5_{H^s},
\end{align*}

\noi
where $C>0$ is an absolute constant and $C_{\ast}>0$ is a constant depending only
on the $H^{1/2}_+(\R)$-norm of $W_0$.
\end{lemma}

\begin{proof}
Since $W\in L^2_+(\R)$ and using Lemma \ref{phase 0 R}, we have that
\begin{align*}
\widehat{f_{\textup{osc}}}&(W(t),s,\xi )\\
&=-i\int\int_{\phi\neq 0}e^{is\phi (\xi,\eta,\zeta)}
\hat{W}(t,\eta-\zeta )\hat{W}(t,\zeta )\cj{\hat{W}}(t,\eta-\xi )\pmb{1}_{\eta-\zeta\geq 0}
\pmb{1}_{\zeta\geq 0}\pmb{1}_{\eta-\xi\geq 0}d\eta d\zeta \\
&=-i\pmb{1}_{\xi<0}\int\int e^{is\phi(\xi,\eta,\zeta )}\hat{W}(t,\eta-\zeta )\hat{W}(t,\zeta )
\cj{\hat{W}}(t,\eta-\xi)\pmb{1}_{\eta\geq \zeta}
\pmb{1}_{\zeta\geq 0}d\eta d\zeta.
\end{align*}

\noi
Then,
\begin{align*}
\widehat{F_{\textup{osc}}}(W(t),&t,\xi )=\int_0^t \widehat{f_{\textup{osc}}}(W(t),s,\xi )ds\\
&=-i\pmb{1}_{\xi<0}\int\int \frac{e^{it\phi(\xi,\eta,\zeta )}-1}{i\phi}\hat{W}(t,\eta-\zeta )
\hat{W}(t,\zeta )\cj{\hat{W}}(t,\eta-\xi)\pmb{1}_{\eta\geq \zeta}
\pmb{1}_{\zeta\geq 0}d\eta d\zeta.
\end{align*}

\noi
Notice that in the region $\xi<0$ and $\{(\eta,\zeta)\in\R^{2}\big|\eta\geq \zeta\geq 0\}$, we have that
\[\phi(\xi,\eta,\zeta)=|\xi|-|\zeta|+|\eta-\xi|-|\eta-\zeta|=-\xi-\zeta+\eta-\xi-\eta+\zeta=-2\xi.\]

\noi
Then,
\begin{align}\label{eqn:Fosc}
\widehat{F_{\textup{osc}}}(W(t),t,\xi )
=\frac{e^{-2it\xi}-1}{2\xi}\mathcal{F}(|W|^2W)(\xi)\pmb{1}_{\xi<0}.
\end{align}

\noi
We now compute the $L^2$-norm of $F_{\textup{osc}}(W(t),t)$, using Parseval's identity:
\begin{align*}
2\pi\|F_{\textup{osc}}(W(t),t)\|_{L^2(\R)}^2&=\|\widehat{F_{\textup{osc}}}(W(t),t)\|_{L^2(\R)}^2
=\int_{-\infty}^0\frac{\sin^2(t\xi)}{\xi^2}\big|\mathcal{F}(|W|^2W)(\xi)\big|^2 d\xi\\
&\leq\big\|\mathcal{F}(|W|^2W)\big\|^2_{L^{\infty}(\R)}\int_{-\infty}^0\frac{\sin^2(t\xi)}{\xi^2}d\xi\\
&\leq\||W|^2W\|^2_{L^{1}(\R)}t\int_{0}^{\infty}\frac{\sin^2x}{x^2}dx\\
&\leq Ct\|W\|_{L^3(\R)}^6\leq Ct\|W(t)\|_{H^{1/2}_+(\R)}^6\leq Ct\|W_0\|_{H^{1/2}_+(\R)}^6.
\end{align*}

\noi
The last inequality is due to the conservation of the $H^{1/2}_+$-norm
by the flow of the Szeg\"o equation.
Therefore,
\begin{align*}
\|F_{\textup{osc}}(W(t),t)\|_{L^2(\R)}\leq C_{\ast}t^{1/2} \text{ for all } t\in\R.
\end{align*}

\noi
Let us now estimate the $\dot{H}^s$-norm of $F_{\textup{osc}}(W(t),t)$ for $s\geq 1$.
\begin{align*}
\|F_{\textup{osc}}(W(t),t)\|_{\dot{H}^s(\R)}^2&=
 \int_{-\infty}^0\xi^{2s}\frac{\sin^2(t\xi)}{\xi^2}|\mathcal{F}(|W|^2W)(\xi)\big|^2d\xi\\
&\leq\int_{-\infty}^0\xi^{2(s-1)}|\mathcal{F}(|W|^2W)(\xi)\big|^2d\xi\\
&\leq \||W|^2W\|^2_{\dot{H}^{s-1}(\R)}\leq \||W|^2W\|^2_{H^{s}(\R)}\leq \|W\|^6_{H^{s}(\R)}.
\end{align*}

\noi
Therefore,
\begin{align*}
\|F_{\textup{osc}}(W(t),t)\|_{H^s(\R)}^2\leq C_{\ast}t^{1/2}+C\|W\|^3_{H^{s}(\R)}.
\end{align*}

\noi
We proceed similarly for $D_{W}F_{\textup{osc}}(W(t),t)\cdot f_{\textup{res}}(W)$.
First, we notice that
\begin{align*}
\mathcal{F}\big(D_{W}F_{\textup{osc}}(W,t)\cdot f_{\textup{res}}(W)\big)(\xi)
=&2\frac{e^{-2it\xi}-1}{2\xi}\mathcal{F}(|W|^2f_{\textup{res}}(W))(\xi)\pmb{1}_{\xi<0}\\
&+\frac{e^{-2it\xi}-1}{2\xi}\mathcal{F}(W^2\cj{f_{\textup{res}}(W)})(\xi)\pmb{1}_{\xi<0}.
\end{align*}

\noi
We use in what follows the fact that $f_{\textup{res}}(W)=\Pi_+(|W|^2W)$, which is a
consequence of equation \eqref{fres R} and of $W\in L^2_+(\R)$.
We estimate the $L^2$-norm, using Parseval's identity:
\begin{align*}
2\pi\|D_{W}F_{\textup{osc}}&(W,t)\cdot f_{\textup{res}}(W)\|_{L^2(\R)}^2
=\|\mathcal{F}\big( D_{W}F_{\textup{osc}}(W,t)\cdot f_{\textup{res}}(W)\big)(\xi)\|_{L^2(\R)}^2\\
&\leq C\int_{-\infty}^0\frac{\sin^2(t\xi)}{\xi^{2}}\big|\mathcal{F}(|W|^2f_{\textup{res}}(W)(\xi)\big|^2 d\xi\\
&+C\int_{-\infty}^0\frac{\sin^2(t\xi)}{\xi^{2}}\big|\mathcal{F}(W^2\cj{f_{\textup{res}}(W)}(\xi)\big|^2 d\xi\\
&\leq C\Big(\big\|\mathcal{F}(|W|^2f_{\textup{res}}(W)\big\|^2_{L^{\infty}(\R)}
+\big\|\mathcal{F}(W^2\cj{f_{\textup{res}}(W)})\big\|^2_{L^{\infty}(\R)}\Big)\int_{-\infty}^0\frac{\sin^2(t\xi)}{\xi^{2}}d\xi\\
&\leq C\big(\||W|^2f_{\textup{res}}(W)\|^2_{L^{1}(\R)}+\|W^2\cj{f_{\textup{res}}(W)}\|^2_{L^{1}(\R)}\big)t\int_{0}^{\infty}\frac{\sin^2x}{x^{2}}dx\\
&\leq Ct\|W\|_{L^4(\R)}^4\|f_{\textup{res}}(W)\|_{L^2(\R)}^2 \leq Ct\|W\|_{L^4(\R)}^4\|\Pi_+(|W|^2W)\|_{L^2(\R)}^2 \\
&\leq Ct\|W\|_{L^4(\R)}^4\|W\|_{L^6(\R)}^6\leq Ct\|W(t)\|_{H^{1/2}_+(\R)}^{10}\leq C\|W_0\|_{H^{1/2}_+(\R)}^{10}t\leq C_{\ast}t.
\end{align*}

\noi
Then, proceeding as in the case of $F_{\textup{osc}}(W)$ and using the structure of an algebra of $H^{s}$,
$s\geq 1$, we have that
\begin{align*}
\|D_{W}F_{\textup{osc}}(W,t)\cdot f_{\textup{res}}(W)\|_{\dot{H}^s(\R)}^2&
\leq C\||W|^2f_{\textup{res}}(W)\|^2_{\dot{H}^{s-1}(\R)}+C\|W^2\cj{f_{\textup{res}}(W)}\|^2_{\dot{H}^{s-1}(\R)}\\
& \leq C\|W\|_{H^s(\R)}^4\|f_{\textup{res}}(W)\|_{H^s(\R)}^2\\
&\leq C\|W\|_{H^s(\R)}^4\||W|^2W\|_{H^s(\R)}^2\leq \|W\|_{H^s(\R)}^{10}.
\end{align*}

\noi
Therefore, for $s\geq 1$ we have
\begin{align*}
\|D_{W}F_{\textup{osc}}(W,t)\cdot f_{\textup{res}}(W)\|_{H^s(\R)}&
\leq C_{\ast}t^{1/2}+C\|W\|_{H^s(\R)}^{5}.
\end{align*}
\end{proof}

\subsection{Proof of Theorem \ref{Main theorem}}

\begin{proof}[Proof of Theorem \ref{Main theorem}]
Let $v$ be the solution of equation \eqref{NLW main thm}.
With the change of variables $u(t)=\frac{1}{\eps}e^{i|D|t}v(t)$, we have that
 $u$ satisfies the equation \eqref{eqn u}. By the Duhamel formula, it follows that
\begin{equation}\label{Duhamel u}
u(t)=W_0+\eps^2\int_0^tf(u(s),s)ds,
\end{equation}

\noi
Set $w(t):=u(t)-u_{\textup{app}}(t)$, where $u_{\textup{app}}$
is defined by \eqref{def uapp}. By equations \eqref{Duhamel u} and \eqref{Duhamel u_app}, we have that
\begin{align*}
w(t)=&\eps^2\int_0^t\big(f(u(s),s)-f(u_{\textup{app}}(s),s)\big)ds
-\int_0^t R_{\eps}(W(s),s)ds\\
=&\eps^2\int_0^t\big(f(u(s),s)-f(u_{\textup{app}}(s),s)\big)ds
-\eps^2\int_0^t\Big(f(W(s),s)-f(u_{\textup{app}}(s),s)\Big)ds\\
&-\eps^4 \int_0^tD_{W}F_{\textup{osc}}(W(s),s)\cdot f_{\textup{res}}(W(s))ds=\mathrm{I+II+III}.
\end{align*}

\noi
Here $W$ denotes the solution of the renormalization group equation \eqref{eq:W}.
In what follows, we estimate each of the terms I,II,III in the $H^s$-norm, $s> 1/2$.
Using the definition of $u_{\textup{app}}$ \eqref{def uapp},
 and the estimates in Lemma \ref{Est 1 R}, it follows that
\begin{align*}
\|u_{\textup{app}}(t)\|_{H^s}\leq \|W\|_{H^s}+\eps^2\|F_{\textup{osc}(W,t)}\|_{H^s}\leq \|W\|_{H^s}+\eps^2C_{\ast}t^{1/2}+\eps^2C\|W\|_{H^s}^3.
\end{align*}

\noi
Then, we have
\begin{align*}
\|\mathrm{I}\|_{H^s}&\leq \eps^2\int_0^t\|w(\tau)\|_{H^s}\big(\|u(\tau)\|_{H^s}^2
+\|u_{\textup{app}}(\tau)\|_{H^s}^2\big)d\tau\\
&\leq C \eps^2\int_0^t\|w(\tau)\|_{H^s}\big(\|w(\tau)\|_{H^s}^2
+\|u_{\textup{app}}(\tau)\|_{H^s}^2\big)d\tau\\
&\leq C \eps^2\int_0^t\|w(\tau)\|_{H^s}\big(\|w(\tau)\|^2_{H^s}+\|W\|_{H^s}^2+\eps^4C_{\ast}t+\eps^4C\|W\|_{H^s}^6\big)d\tau.
\end{align*}

\noi
Using $W(s)-u_{\textup{app}}(s)=-\eps^2F_{\textup{osc}}(W(s),s)$,
and proceeding as above, we obtain
\begin{align*}
\|\mathrm{II}\|_{H^s}&\leq  \eps^4t\|F_{\textup{osc}}(t,W(t))\|_{L^{\infty}([0,t],H^s)}
\big(\|W\|_{L^{\infty}([0,t],H^s)}^2
+\|u_{\textup{app}}\|_{L^{\infty}([0,t],H^s)}^2\big)\\
&\leq C_{\ast} \eps^4t(t^{1/2}+\|W\|_{L^{\infty}([0,t],H^s)}^3)(\|W\|_{L^{\infty}([0,t],H^s)}^2+\eps^4t+\eps^4\|W\|_{L^{\infty}([0,t],H^s)}^6).
\end{align*}

\noi
and
\begin{align*}
\|\mathrm{III}\|_{H^s}&\leq C_{\ast}\eps^4t(t^{1/2}+\|W\|_{L^{\infty}([0,t],H^s)}^5).
\end{align*}

\noi
In order to estimate $w$ we will use a bootstrap argument. Let $0\leq\alpha\leq \frac{1}{2}$, $\delta>0$ small enough, and set
\begin{align}\label{def T R}
T:=\sup\Big\{t\geq 0\Big|\|w(t)\|_{H^s}\leq 1\Big\}.
\end{align}

\noi
We will prove that $T> \frac{1}{\eps^2}\Big(\log(\frac{1}{\eps^{\delta}})\Big)^{1-2\alpha}$. Suppose by contradiction that
\begin{equation}\label{assumption T}
T\leq\frac{1}{\eps^2}\Big(\log(\frac{1}{\eps^{\delta}})\Big)^{1-2\alpha}.
\end{equation}

\noi
According to the hypothesis on $\W$ and since $\W=\eps W$,
we have that $\|W(t)\|_{H^s(\R)}\leq C\Big(\log(\frac{1}{\eps^{\delta}})\Big)^{\alpha}$
for all $t\in\R$. Using the estimates of $\mathrm{I,II,III}$, we obtain
for $0\leq t\leq\frac{1}{\eps^2}\Big(\log(\frac{1}{\eps^{\delta}})\Big)^{1-2\alpha}$ that
\begin{align*}
\|w(t)\|_{H^s}\leq &C\eps^2\int_0^t\|w(\tau)\|_{H^s}(1+\|W\|_{H^s}^2+\eps^4C_{\ast}\tau+\eps^4C\|W\|_{H^s}^6)d\tau\\
&+
C_{\ast} \eps^4t(t^{1/2}+\|W\|_{H^s}^3)(\|W\|_{H^s}^2+\eps^4t+\eps^4\|W\|_{H^s}^6)+C_{\ast}\eps^4t(t^{1/2}+\|W\|_{H^s}^5)\\
\leq & C\eps^2\Big(\log(\frac{1}{\eps^{\delta}})\Big)^{2\alpha}\int_0^t\|w(\tau)\|_{H^s}d\tau+
C_{\ast} \eps\Big(\log(\frac{1}{\eps^{\delta}})\Big)^{\frac{3}{2}(1-2\alpha)}\Big(\log(\frac{1}{\eps^{\delta}})\Big)^{2\alpha}\\
&+C_{\ast}\eps\Big(\log(\frac{1}{\eps^{\delta}})\Big)^{\frac{3}{2}(1-2\alpha)}.
\end{align*}

\noi
By Gronwall's inequality it follows, for $0\leq t\leq\frac{1}{\eps^2}\Big(\log(\frac{1}{\eps^{\delta}})\Big)^{1-2\alpha}$, that
\[\|w(t)\|_{H^s}\leq C_{\ast}\eps\Big(\log(\frac{1}{\eps^{\delta}})\Big)^{\frac{3}{2}-\alpha}e^{C\log(\frac{1}{\eps^{\delta}})}
\leq C_{\ast}\eps\Big(\log(\frac{1}{\eps^{\delta}})\Big)^{\frac{3}{2}-\alpha}\frac{1}{\eps^{C\delta}}\leq C_{\ast}\eps^{1-C_0\delta},\]

\noi
If $\delta$ is sufficiently small, this bound is much better
than the one imposed in the definition of $T$. Since $w$ is continuous with respect to $t$,
it follows that there exists $\gamma>0$ such that
\[\|w(t)\|_{H^s}\leq 1,\]

\noi
for $0\leq t\leq \frac{1}{\eps^2}\Big(\log(\frac{1}{\eps^{\delta}})\Big)^{1-2\alpha}+\gamma$. This
contradicts the assumption \eqref{assumption T} we made on $T$. Therefore, $T> \frac{1}{\eps^2}\Big(\log(\frac{1}{\eps^{\delta}})\Big)^{1-2\alpha}$
and, moreover, $\|w(t)\|_{H^s}\leq \eps^{1-C_0\delta}$ for all $0\leq t\leq \frac{1}{\eps^2}\Big(\log(\frac{1}{\eps^{\delta}})\Big)^{1-2\alpha}$.
This yields
\[\|u(t)-W(t)-\eps^2F_{\textup{osc}}(W(t),t)\|_{H^s(\R)}\leq C_{\ast}\eps^{1-C_0\delta}\]

\noi
for all $0\leq t\leq \frac{1}{\eps^2}\Big(\log(\frac{1}{\eps^{\delta}})\Big)^{1-2\alpha}$.
Since by Lemma \ref{Est 1 R}, we have that
\[\|\eps^2F_{\textup{osc}}(W(t),t)\|_{H^s(\R)}\leq \eps^2(C_{\ast}t^{1/2}+C\|W\|_{H^s}^3)
\leq C_{\ast}\eps\Big(\log(\frac{1}{\eps^{\delta}})\Big)^{\frac{1}{2}(1-2\alpha)}
\leq C_{\ast}\eps^{1-C_0\delta}\]

\noi
for $0\leq t\leq\frac{1}{\eps^2}\Big(\log(\frac{1}{\eps^{\delta}})\Big)^{1-2\alpha}$, we obtain
\[\|u(t)-W(t)\|_{H^s(\R)}\leq C_{\ast}\eps^{1-C_0\delta}.\]

\noi
Recalling that $u(t)=\frac{1}{\eps}e^{i|D|t}v(t)$ and $W=\frac{1}{\eps}\mathcal{W}$, we obtain that
\begin{equation}
\|v(t)-e^{-i|D|t}\mathcal{W}(t)\|_{H^s(\R)}\leq C_{\ast}\eps^{2-C_0\delta},
\end{equation}

\noi
for $0\leq t\leq\frac{1}{\eps^2}\Big(\log(\frac{1}{\eps^{\delta}})\Big)^{1-2\alpha}$.
\end{proof}

\subsection{Proof of Corollary \ref{Cor}}

\begin{proof}[Proof of Corollary \ref{Cor}]
Let $W$ be the solution of the equation
\begin{equation*}
\begin{cases}
i\partial_t W=\eps^2\Pi_+(|W|^2W)\\
W(0)=W_0.
\end{cases}
\end{equation*}

\noi
With the change of variables $W(t,x)=y(\eps^2t,x)$,
we have that $y$ satisfies the Szeg\"o equation:
\begin{equation*}
\begin{cases}
i\partial_ty=\Pi_+(|y|^2y)\\
y(0)=W_0.
\end{cases}
\end{equation*}

\noi
Then, according to Proposition \ref{prop Szego},
we have that $\|y(t)\|_{H^s(\R)}\sim t^{2s-1}$,
for all $s>\frac{1}{2}$ and for $t>1$ sufficiently large. Consequently, we have
\[\|W(t)\|_{H^s(\R)}\sim (\eps^2t)^{2s-1}\]

\noi
for $\eps^2t$ sufficiently large.
Suppose $\frac{1}{2\eps^2}\Big(\log(\frac{1}{\eps^{\delta}})\Big)^{\frac{1}{4s-1}}\leq t\leq \frac{1}{\eps^2}\Big(\log(\frac{1}{\eps^{\delta}})\Big)^{\frac{1}{4s-1}}$. Then,
\begin{equation}\label{1 nlw}
\frac{c}{2^{2s-1}}\Big(\log(\frac{1}{\eps^{\delta}})\Big)^{\frac{2s-1}{4s-1}}\leq \|W(t)\|_{H^s(\R)}
\leq C\Big(\log(\frac{1}{\eps^{\delta}})\Big)^{\frac{2s-1}{4s-1}}.
\end{equation}

\noi
Applying Theorem \ref{Main theorem} with $\alpha=\frac{2s-1}{4s-1}\in (0,\frac{1}{2})$, we obtain that
\begin{equation}\label{2 nlw}
\|v(t)-e^{-i|D|t}\eps W(t)\|_{H^s(\R)}\leq C_{\ast}\eps^{2-C_0\delta},
\end{equation}

\noi
for $0 \leq t\leq \frac{1}{\eps^2}\Big(\log(\frac{1}{\eps^{\delta}})\Big)^{\frac{1}{4s-1}}$.
Then, equations \eqref{1 nlw} and \eqref{2 nlw} yield
\begin{align*}
\|v(t)\|_{H^s}&\geq \Big|\|\eps W(t)\|_{H^s}-\|v(t)-e^{-i|D|t}\eps W(t)\|_{H^s}\Big|\\
&\geq \frac{c}{2^{2s-1}}\eps\Big(\log(\frac{1}{\eps^{\delta}})\Big)^{\frac{2s-1}{4s-1}}-C_{\ast}\eps^{2-C_0\delta}
\geq C\eps \Big(\log(\frac{1}{\eps^{\delta}})\Big)^{\frac{2s-1}{4s-1}}.
\end{align*}

\noi
Since $v(0)=\eps W_0$, it follows that,
for $\frac{1}{2\eps^2}\Big(\log(\frac{1}{\eps^{\delta}})\Big)^{\frac{1}{4s-1}}
\leq t\leq \frac{1}{\eps^2}\Big(\log(\frac{1}{\eps^{\delta}})\Big)^{\frac{1}{4s-1}}$,
we have
\begin{align*}
\frac{\|v(t)\|_{H^s}}{\|v(0)\|_{H^s}}\geq C\Big(\log(\frac{1}{\eps^{\delta}})\Big)^{\frac{2s-1}{4s-1}}.
\end{align*}

\end{proof}

\section{First order approximation for the \eqref{eq:dirac} \,\,equation on $\T$}

\subsection{The renormalization group equation for the case of $\T$}
We decompose a $2\pi$-periodic function $a(t)$ in the following way:
\begin{equation}\label{eq:a(t)}
a(t)=a_{\textup{res}}+a_{\textup{osc}}(t),
\end{equation}

\noi
where
\begin{equation}\label{eq:a(t)res}
a_{\textup{res}}=\frac{1}{2\pi}\int_0^{2\pi}a(\tau)d\tau
\end{equation}

\noi
is the mean of the function $a(t)$ or equivalently, the Fourier coefficient at zero. The oscillatory part is then
\begin{equation}\label{eq:a(t)osc}
a_{\textup{osc}}(t)=\sum_{k\neq 0}\widehat{a}(k)e^{itk}.
\end{equation}

\noi
With this decomposition, we notice that for the torus, the resonant and non-resonant part of the nonlinearity are the following:

\begin{align*}
f_{\textup{res}}(u,x)&=-i\sum_{k=-\infty}^{\infty}e^{ikx}
\sum_{\substack{k-l+m-j=0\\ |k|-|l|+|m|-|j|=0}}\hat{u}(j)\hat{u}(l)\cj{\hat{u}}(m),\\
f_{\textup{osc}}(u,s,x)&=-i\sum_{k=-\infty}^{\infty}e^{ikx}
\sum_{\substack{k-l+m-j=0\\ |k|-|l|+|m|-|j|\neq 0}}e^{is(|k|-|l|+|m|-|j|)}
\hat{u}(j)\hat{u}(l)\cj{\hat{u}}(m).
\end{align*}

\noi
A slight difference with the case of $\R$ is the definition of $F_{\textup{osc}}(u,t)$:
\[F_{\textup{osc}}(u,t,x):=-i\sum_{k=-\infty}^{\infty}e^{ikx}
\sum_{\substack{k-l+m-j=0\\ |k|-|l|+|m|-|j|\neq 0}}\frac{e^{it(|k|-|l|+|m|-|j|)}}{i(|k|-|l|+|m|-|j|)}
\hat{u}(j)\hat{u}(l)\cj{\hat{u}}(m),\]

\noi
whereas for $\R$, we had $F_{\textup{osc}}(u,t)=\int_0^tf_{\textup{osc}}(u,s)ds$.
Notice that in both cases we have that $\frac{\partial F_{\textup{osc}}}{\partial t}(u,t,x)=f_{\textup{osc}}(u,t,x)$.

As it was shown in \cite{PGSG res}, the following lemma holds:

\begin{lemma}\label{Resonances T}
We have that $k-l+m-j=0$ and $|k|-|l|+|m|-|j|=0$ if and only if we are in one of the following cases:

$\mathrm{(i)}$ If $k>0$ and $\{l,m,j\geq 0\}\cup\{k=l\}\cup\{k=j\}$

$\mathrm{(ii)}$ If $k=0$ and $\{l,m,j\geq 0\}\cup\{l,m,j\leq 0\}$

$\mathrm{(iii)}$ If $k<0$ and $\{l,m,j\leq 0\}\cup\{k=l\}\cup\{k=j\}$.
\end{lemma}

\noi
We decompose the region where $k-l+m-j=0$ and $|k|-|l|+|m|-|j|=0$
into disjoint sub-regions,
and we compute the Fourier transform of the resonant part $f_{\textup{res}}(u)$. We obtain the following ten terms:

1. The case $k,l,m,j\geq 0$:
\begin{align*}
&-i\sum_{\substack{k-l+m-j=0\\k,l,m,j\geq 0}}\hat{u}(j)\hat{u}(l)\cj{\hat{u}}(m)
=-i\mathcal{F}\big(\Pi_+ (|u_+|^2u_+)\big)(k)\pmb{1}_{k\geq 0}.
\end{align*}

2. The case $k\geq 0$, $k=l$, $m=j<0$:

\begin{align*}
&-i\sum_{\substack{l=k\geq 0\\m=j<0}}\hat{u}(j)\hat{u}(l)\cj{\hat{u}}(m)
=-i\hat{u}(k)\pmb{1}_{k\geq 0}\sum_{j=-\infty}^{-1}|\hat{u}(j)|^2
=-i\|u_-\|_{L^2}^2\hat{u}_+(k)\pmb{1}_{k\geq 0}.
\end{align*}

3. The case $k\geq 0$, $k=j$, $m=l<0$. We obtain as above $-i\|u_-\|_{L^2}^2\hat{u}_+(k)\pmb{1}_{k\geq 0}$.

4. The case $k=0$ and $l,m,j< 0$:
\begin{align*}
&-i\sum_{\substack{-l+m-j=0\\l,m,j<0}}\hat{u}(j)\hat{u}(l)\cj{\hat{u}}(m)
=-i\mathcal{F}\big(|u_-|^2u_-\big)(0).
\end{align*}

5. The case $k,l,m,j<0$:
\begin{align*}
&-i\sum_{\substack{k-l+m-j=0\\k,l,m,j< 0}}\hat{u}(j)\hat{u}(l)\cj{\hat{u}}(m)
=-i\mathcal{F}\big(\Pi_- (|u_-|^2u_-)\big)(k)\pmb{1}_{k< 0}.
\end{align*}

6. The case $k<0$, $l=0$, $j<0$, $m<0$:
\begin{align*}
-i\sum_{\substack{k+m-j=0, l=0,\\k,m,j< 0}}\hat{u}(j)\hat{u}(l)\cj{\hat{u}}(m)
&=-i\hat{u}(0)\pmb{1}_{k< 0}\sum_{j=k-1}^{-1}\hat{u}(j)\cj{\hat{u}}(j-k)\\
&=-i\hat{u}(0)\mathcal{F}\big(\Pi_-|u_-|^2)\big)(k)\pmb{1}_{k< 0}.
\end{align*}

7. The case $k<0$, $j=0$, $l<0$, $m<0$. We obtain as above
$-i\hat{u}(0)\mathcal{F}\big(\Pi_-|u_-|^2)\big)(k)\pmb{1}_{k< 0}$.

8. The case $k<0$, $m=0$, $l<0$, $j<0$:
\begin{align*}
-i\sum_{\substack{k-l-j=0, m=0,\\k,l,j< 0}}\hat{u}(j)\hat{u}(l)\cj{\hat{u}}(m)
&=-i\cj{\hat{u}}(0)\pmb{1}_{k< 0}\sum_{j=k-1}^{-1}\hat{u}(j)\hat{u}(k-j)\\
&=-i\cj{\hat{u}}(0)\mathcal{F}\big(u_-^2\big)(k)\pmb{1}_{k< 0}.
\end{align*}

9. The case $k<0$, $k=l$, $m=j\geq 0$:
\begin{align*}
&-i\sum_{\substack{k=l<0, \\ m=j\geq 0}}\hat{u}(j)\hat{u}(l)\cj{\hat{u}}(m)
=-i\hat{u}(k)\pmb{1}_{k< 0}\sum_{j=0}^{\infty}|\hat{u}(j)|^2=-i\|u_+\|_{L^2}^2\hat{u}_-(k)\pmb{1}_{k< 0}.
\end{align*}

10. The case $k<0$, $k=j$, $l=m\geq 0$. We obtain as above $-i\|u_+\|_{L^2}^2\hat{u}_-(k)\pmb{1}_{k< 0}$.

Thus, the resonant part of the nonlinearity is
\begin{align}\label{f_res T}
f_{\textup{res}}(u,x)=&-i\Pi_+(|u_+|^2u_+)
-2i\|u_-\|_{L^2}^2u_+
-i\mathcal{F}\big(|u_-|^2u_-\big)(0)\\
&-i\Pi_-(|u_-|^2u_-)
-2i\hat{u}(0)\Pi_-(|u_-|^2)
-i\cj{\hat{u}}(0)u_-^2
-2i\|u_+\|_{L^2}^2u_-.\notag
\end{align}

\begin{lemma}\label{Prop uniq}
Let $s>\frac{1}{2}$ and $W_0\in H^s_+(\T)$. We consider the renormalization group equation:
\begin{align}\label{RG of order 2}
\begin{cases}
\partial_tu=\eps^2f_{\textup{res}}(u)\\
u(0)=W_0
\end{cases}
\end{align}

\noi
This equation has a unique global solution in $H^s(\T)$ which coincides
with $W\in C(\R,H^s(\T))$, the solution of the following equation
\begin{align}\label{eqn W T}
\begin{cases}
i\partial_tW=\eps^2\Pi_+(|W|^2W)\\
W(0)=W_0
\end{cases}
\end{align}

\noi
In particular, $u_-(t)=0$ for all $t\in\R$.
\end{lemma}

\begin{proof}
We first notice that $f_{\textup{res}}:H^s(\T)\to H^s(\T)$, $s>\frac{1}{2}$, defined in equation
\eqref{f_res T} is a locally Lipschitz mapping. Indeed, one can prove using the structure of algebra of $H^s(\T)$, that
\[\|f_{\textup{res}}(u)-f_{\textup{res}}(v)\|_{H^s}\leq \|u-v\|_{H^s}(\|u\|_{H^s}^2+\|v\|_{H^s}^2),\]

\noi
for all $u,v\in H^s(\T)$. Then, by the Cauchy-Lipschitz theorem it follows that
equation \eqref{RG of order 2} has an unique solution in $H^s(\T)$.

With the change of variables $W(t,x)=y(\eps^2t,x)$, we obtain from
equation \eqref{eqn W T} that $y$ satisfies the Szeg\"o equation \eqref{Szego simple}.
The Szeg\"o equation has a unique global solution
supported on non-negative frequencies. Thus $W$ is unique and satisfies $W_-(t)=0$
for all $t\in\R$.
The only term in the expression of $f_{\textup{res}}(u)$ \eqref{f_res T},
which does not contain $u_-$ is $-i\Pi(|u_+|^2u_+)$. Therefore we immediately notice that
the solution of the equation \eqref{eqn W T}
is also the solution of the equation \eqref{RG of order 2}.
\end{proof}

\subsection{Estimates for the oscillatory part of the nonlinearity in the case of $\T$}

To re-prove Theorem \ref{Thm T} we apply exactly the same method used in the proof
of Theorem \ref{Main theorem}. The only changes that appear are in the estimate of
$F_{\textup{osc}}(W(t),t)$. We show that on $\T$ we obtain a better estimate
than on $\R$.

\begin{lemma}\label{Est T 1}
Let $s>\frac{1}{2}$. For all $W\in H_+^s(\T)$, we have that
\begin{align*}
\|F_{\textup{osc}}(W,t)\|_{H^s(\T)}\leq &C_s\|W\|^3_{H^s(\T)},\\
\|D_{W}F_{\textup{osc}}(W,t)\cdot f_{\textup{res}}(W)\|_{H^s(\T)}\leq &C_s\|W\|^5_{H^s(\T)},
\end{align*}

\noi
where $C_s$ is a constant depending only on $s$.
\end{lemma}

\begin{proof}
The Fourier coefficients of $F_{\textup{osc}}(W,t)$ are:
\begin{align*}
\mathcal{F}(F_{\textup{osc}})(W,t,k )
=-\sum_{\substack{k-l+m-j=0,\\|k|-|l|+|m|-|j|\neq 0}}\frac{e^{it(|k|-|l|+|m|-|j|)}}{i(|k|-|l|+|m|-|j|)}
\hat{W}(j)\hat{W}(l)\cj{\hat{W}}(m).
\end{align*}

\noi
Setting $\hat{W}_k:=\hat{W}(k)$ for all $k\in\Z$, and using the convexity of the
function $|x|^{\alpha}$ if $\alpha>1$, we have that
\begin{align*}
&\|F_{\textup{osc}}(W,t)\|_{H^s(\T)}^2=\sum_{k\in\Z}(1+|k|^2)^s\Big|
\sum_{\substack{k-l+m-j=0,\\|k|-|l|+|m|-|j|\neq 0}}\frac{e^{it(|k|-|l|+|m|-|j|)}}{|k|-|l|+|m|-|j|}
\hat{W}_j\hat{W}_l\cj{\hat{W}}_m\Big|^2\\
&\leq\sum_{k\in\Z}(1+|k|^2)^s\Big(\sum_{k=l-m+j}
|\hat{W}_j\hat{W}_l\cj{\hat{W}}_m|\Big)^2\\
&\leq\sum_{k\in\Z}(1+|k|^2)^s\sum_{k=l-m+j}
|\hat{W}_j\hat{W}_l\cj{\hat{W}}_m|\sum_{k=\tilde{l}-\tilde{m}+\tilde{j}}
|\hat{W}_{\tilde{j}}\hat{W}_{\tilde{l}}\cj{\hat{W}}_{\tilde{m}}|\\
&\leq \sum_{k\in\Z}(1+|k|^2)^s\sum_{\substack{k=l-m+j,\\k=\tilde{l}-\tilde{m}+\tilde{j}}}
|\hat{W}_j\hat{W}_l\cj{\hat{W}}_m\hat{W}_{\tilde{j}}\hat{W}_{\tilde{l}}\cj{\hat{W}}_{\tilde{m}}|\\
&=\sum_{l-\tilde{l}-m+\tilde{m}+j-\tilde{j}=0}(1+|l-m+j|^2)^{s/2}
(1+|\tilde{l}-\tilde{m}+\tilde{j}|^2)^{s/2}|\hat{W}_j||\hat{W}_l||
\hat{W}_m||\hat{W}_{\tilde{j}}||\hat{W}_{\tilde{l}}||\hat{W}_{\tilde{m}}|\\
&\leq \sum_{l-\tilde{l}-m+\tilde{m}+j-\tilde{j}=0}(1+3(|l|^2+|m|^2+|j|^2))^{s/2}
(1+3(|\tilde{l}|^2+|\tilde{m}|^2+|\tilde{j}|^2))^{s/2}\\
&\hphantom{XXXXXXXX}\times|\hat{W}_j||\hat{W}_l||
\hat{W}_m||\hat{W}_{\tilde{j}}||\hat{W}_{\tilde{l}}||\hat{W}_{\tilde{m}}|\\
&\leq 9\sum_{l-\tilde{l}-m+\tilde{m}+j-\tilde{j}=0}[(1+|l|^2)+(1+|m|^2)+(1+|j|^2)]^{s/2}
\\
&\hphantom{XXXXXXXX}\times [(1+|\tilde{l}|^2)+(1+|\tilde{m}|^2)+(1+|\tilde{j}|^2)]^{s/2}|\hat{W}_j||\hat{W}_l||
\hat{W}_m||\hat{W}_{\tilde{j}}||\hat{W}_{\tilde{l}}||\hat{W}_{\tilde{m}}|\\
&\leq C_s \sum_{l-\tilde{l}-m+\tilde{m}+j-\tilde{j}=0}[(1+|l|^2)^{s/2}+(1+|m|^2)^{s/2}+(1+|j|^2)^{s/2}]
\\
&\hphantom{XXXXXXX}\times[(1+|\tilde{l}|^2)^{s/2}+(1+|\tilde{m}|^2)^{s/2}+(1+|\tilde{j}|^2)^{s/2}]|\hat{W}_j||\hat{W}_l||
\hat{W}_m||\hat{W}_{\tilde{j}}||\hat{W}_{\tilde{l}}||\hat{W}_{\tilde{m}}|\\
&\leq C_s \sum_{l-\tilde{l}-m+\tilde{m}+j-\tilde{j}=0}(1+|j|^2)^{s/2}|\hat{W}_j||\hat{W}_l||
\hat{W}_m|(1+|\tilde{j}|^2)^{s/2}|\hat{W}_{\tilde{j}}||\hat{W}_{\tilde{l}}||\hat{W}_{\tilde{m}}|+\text{ similar terms }\\
\end{align*}

\noi
We consider the functions $V^{\ast}=\sum_{j\in\Z}e^{ixj}\hat{V}_j^{\ast}$ and
$U^{\ast}=\sum_{j\in\Z}e^{ixj}\hat{U}_j^{\ast}$, where
\begin{align*}
\hat{V}_j^{\ast}:=&|\hat{W}_j|\\
\hat{U}_j^{\ast}:=&(1+|j|^2)^{s/2}|\hat{W}_j|.
\end{align*}

\noi
Ignoring the other terms in the above sum, which can be treated in a similar manner as the term we keep,
and using the Sobolev embedding $H^s(\T)\subset L^{\infty}(\T)$ if $s>1/2$, we obtain
\begin{align*}
\|F_{\textup{osc}}(W(t),t)\|_{H^s(\T)}^2&\leq C_s \sum_{l-\tilde{l}-m+\tilde{m}+j-\tilde{j}=0}\hat{U}_j^{\ast}\hat{V}_l^{\ast}
\hat{V}_m^{\ast}\hat{U}_{\tilde{j}}^{\ast}\hat{V}_{\tilde{l}}^{\ast}\hat{V}_{\tilde{m}}^{\ast}
\leq C_s \int_{\T} U^{\ast}\cj{U^{\ast}}(V^{\ast})^2(\cj{V^{\ast}})^2dz\\
&\leq C_s \int_{\T} |U^{\ast}|^2|V^{\ast}|^4dz\leq C_s \|U^{\ast}\|_{L^2(\T)}^2
\|V^{\ast}\|_{L^{\infty}(\T)}^4\\
&\leq C_s \|U^{\ast}\|_{L^2(\T)}^2\|V^{\ast}\|_{H^{s}(\T)}^4
\leq C_s \|V^{\ast}\|_{H^{s}(\T)}^2\|V^{\ast}\|_{H^{s}(\T)}^4\leq C_s \|W\|_{H^{s}(\T)}^6,
\end{align*}

\noi
where $C_s$ denotes a constant depending on $s$.

The second estimate in the statement,
\[\|D_{W}F_{\textup{osc}}(W,t)\cdot f_{\textup{res}}(W)\|_{H^s(\T)}\leq C_s\|W\|^5_{H^s(\T)},\]

\noi
can be proved similarly.
\end{proof}

\subsection{Proof of Theorem \ref{Thm T}}

By the hypothesis we have that $\|W(t)\|_{H^s}\leq C\Big(\log(\frac{1}{\eps^{\delta}})\Big)^{\alpha}$.
Using the definition of $u_{\textup{app}}$ \eqref{def uapp}, and Lemma \ref{Est T 1}, we obtain that
$\|u_{\textup{app}}(t)\|_{H^s}\leq C\Big(\log(\frac{1}{\eps^{\delta}})\Big)^{\alpha}$.
Proceeding as in the proof of Theorem \ref{Main theorem}, we obtain for $0\leq t\leq \frac{1}{\eps^2}\log(\frac{1}{\eps^{\delta}})^{1-2\alpha}$
\begin{align*}
\|w(t)\|_{H^s}&\leq C\Big(\log(\frac{1}{\eps^{\delta}})\Big)^{2\alpha}\eps^2\int_0^t\|w(\tau)\|_{H^s}d\tau
 +C\eps^4\Big(\log(\frac{1}{\eps^{\delta}})\Big)^{5\alpha} t.
\end{align*}

\noi
This yields, by Gronwall's inequality, that for $0\leq t\leq \frac{1}{\eps^2}\Big(\log(\frac{1}{\eps^{\delta}})\Big)^{1-2\alpha}$
we have
\[\|w(t)\|_{H^s}\leq C\eps^4\Big(\log(\frac{1}{\eps^{\delta}})\Big)^{5\alpha}t
e^{C\eps^2\Big(\log(\frac{1}{\eps^{\delta}})\Big)^{2\alpha}t}\leq \eps^{2-C_0\delta},\]

\noi
where $C_0>0$.
Since $w(t)=u(t)-W(t)-\eps^2F_{\textup{osc}}(W(t),t)$ and
\\ $\|F_{\textup{osc}}(W(t),t)\|_{H^s}\leq C\Big(\log(\frac{1}{\eps^{\delta}})\Big)^{3\alpha}$, it follows that
\begin{equation*}
\|u(t)-W(t)\|_{H^s(\T)}\leq C\eps^{2-C_0\delta} \text{ if } 0\leq t\leq \frac{1}{\eps^2}\Big(\log(\frac{1}{\eps^{\delta}})\Big)^{1-2\alpha}.
\end{equation*}

\noi
Then, the changes of variables $v(t)=\eps e^{i|D|t}u(t)$ and $\W=\eps W$ yield the conclusion
\begin{equation*}
\|v(t)-e^{-i|D|t}\W(t)\|_{H^s(\T)}\leq C\eps^{3-C_0\delta} \text{ if } 0\leq t\leq \frac{1}{\eps^2}\Big(\log(\frac{1}{\eps^{\delta}})\Big)^{1-2\alpha}.
\end{equation*}

\section{Second order approximation for the \eqref{eq:dirac} \,\,equation on $\T$}

\subsection{The averaging method at order two}

As before, in the \eqref{eq:dirac} equation with
initial condition $v(0)=\W_0=\eps W_0$,
 we make the change of variables $u(t)=\frac{1}{\eps}e^{i|D|t}v(t)$.
Then $u$ satisfies the equation:
\begin{equation*}
\begin{cases}
\partial_tu=-i\eps^2 e^{i|D|t}(|e^{-i|D|t}u|^2e^{-i|D|t}u)=:f(u,t)\\
u(0)=W_0.
\end{cases}
\end{equation*}

\noi
The averaging method at order two introduced by
Temam and Wirosoetisno in \cite{Temam Wirosoetisno}, consists in
considering the following averaging ansatz:
\begin{equation}\label{def uapp 2}
u_{\textup{app}}(t)=W(t)+\eps^2 N_1(W,t)+\eps^4N_2(W,t)=:N(W,t,\eps),
\end{equation}

\noi
where $W$ is a solution of  the following averaged equation:
\begin{equation}\label{eqn W 2}
\begin{cases}
\partial_tW=\eps^2 R_1(W)+\eps^4 R_2(W)=:R(W,\eps)\\
W(0)=W_0.
\end{cases}
\end{equation}

\noi
The use of these notations is explained by the fact that $R_1,R_2$ turn out to be resonant terms,
while $N_1,N_2$ are non-resonant (oscillatory) terms.

A formal computation then shows that
\begin{align*}
\partial_tu_{\textup{app}}(t)=&\frac{\partial N(W,t,\eps)}{\partial t}=
N'(W,t,\eps)\cdot \frac{\partial W}{\partial t}+\frac{\partial N}{\partial t}(W,t,\eps)\\
=&(\eps^2N_1'(W,t)+\eps^4N_2'(W,t))\cdot (\eps^2 R_1(W)+\eps^4 R_2(W))+\frac{\partial W}{\partial t}\\
&+\eps^2\frac{\partial N_1}{\partial t}(W,t)+\eps^4\frac{\partial N_2}{\partial t}(W,t)\\
=&\eps^2\Big(R_1(W)+\frac{\partial N_1}{\partial t}(W,t)\Big)\\
&+\eps^4\Big(R_2(W)+N_1'(W,t)\cdot R_1(W)
+\frac{\partial N_2}{\partial t}(W,t)\Big)+O(\eps^6).
\end{align*}

\noi
We now formally Taylor-expand $f(u_{\textup{app}}(t),t)$ around $W(t)$,
\begin{align*}
f(u_{\textup{app}},t)&=f(W,t)+f'(W,t)\cdot (u_{\textup{app}}-W)+O(\eps^4)\\
&=f(W,t)+\eps^2f'(W,t)\cdot N_1(W,t)+O(\eps^4).
\end{align*}

\noi
We replace the two expansions into the equation
\begin{align*}
\partial_tu_{\textup{app}}=\eps^2f(u_{\textup{app}},t)+O(\eps^6)
\end{align*}

\noi
in order to determine $R_1,R_2,N_1,N_2$ which yield an approximate solution.
Identifying the coefficients according to the powers of $\eps$, we obtain the equations
\begin{align}\label{eqn R,N}
R_1(W)+\frac{\partial N_1}{\partial t}(W,t)=&f(W,t)\\
R_2(W)+N_1'(W,t)\cdot R_1(W)+\frac{\partial N_2}{\partial t}(W,t)=&f'(W,t)\cdot N_1(W,t)
\end{align}

\noi
Thus, $R_1$ is the part of $f(W,t)$ which does not explicitly depend on $t$. According
to the decomposition given in equations \eqref{eq:a(t)}, \eqref{eq:a(t)res}, and \eqref{eq:a(t)osc}, we have:
\begin{align*}
R_1(W)=f_{\textup{res}}(W) \,\,\,\,\,\,\,\,\,\text{ and } \,\,\,\,\,\,\,\,\,N_1(W,t)=F_{\textup{osc}}(W,t).
\end{align*}

\noi
Then, from the second equation we have:
\begin{align}\label{eqn R_2,N_2}
R_2(W)=&\{f'(W,t)\cdot N_1(W,t)\}_{\textup{res}}-\{N_1'(W,t)\cdot R_1(W)\}_{\textup{res}}\\
\frac{\partial N_2}{\partial t}(W,t)=&\{f'(W,t)\cdot N_1(W,t)\}_{\textup{osc}}-\{N_1'(W,t)\cdot R_1(W)\}_{\textup{osc}}.\notag
\end{align}

\noi
Replacing $R_1,N_1$ and noticing that $F'_{\textup{osc}}(W,t)\cdot f_{\textup{res}}(W)$
does not have a resonant part, we obtain:
\begin{align}\label{eqn R_2,N_2 explicit}
R_2(W)=&\{f'(W,t)\cdot F_{\textup{osc}}(W,t)\}_{\textup{res}}-\{F'_{\textup{osc}}(W,t)
\cdot f_{\textup{res}}(W)\}_{\textup{res}}\\
=&\{f'(W,t)\cdot F_{\textup{osc}}(W,t)\}_{\textup{res}}\notag\\
\frac{\partial N_2}{\partial t}(W,t)=&\{f'(W,t)\cdot F_{\textup{osc}}(W,t)\}_{\textup{osc}}
-F'_{\textup{osc}}(W,t)\cdot f_{\textup{res}}(W).\notag
\end{align}

We set $w(t):=u(t)-u_{\textup{app}}(t)$. In what follows, we
determined a simplified version of the equation
satisfied by $w$. First, by the definition of $u_{\textup{app}}$ \eqref{def uapp 2},
we have that $w$ satisfies:
\begin{align*}
\begin{cases}
\frac{\partial w}{\partial t}=\eps^2 f(u,t)-\frac{\partial W}{\partial t}
-\eps^2\frac{\partial N_1}{\partial t}(W,t)-\eps^2N_1'(W,t)\cdot\frac{\partial W}{\partial t}
-\eps^4\frac{\partial N_2}{\partial t}(W,t)-\eps^4N_2'(W,t)\cdot\frac{\partial W}{\partial t}\\
w(0)=0
\end{cases}
\end{align*}

\noi
We consider the following Taylor expansion of $f(u)$ around $W$:
\begin{align*}
f(u,t)=&f(w+u_{\textup{app}})=f(w+W+\eps^2N_1+\eps^4N_2,t)\\
=&f(W,t)+f'(W,t)\cdot (w+\eps^2N_1+\eps^4N_2)\\
&+\int_0^1f''(\alpha(w+\eps^2N_1+\eps^4N_2)+W)\cdot (w+\eps^2N_1+\eps^4N_2)\\
&\hphantom{XXXXXXXXXXXXX}\otimes(w+\eps^2N_1+\eps^4N_2)(1-\alpha)d\alpha.
\end{align*}

\noi
Replacing this into the equation of $w$ and using the equation \eqref{eqn W 2}, we obtain that
\begin{align*}
\frac{\partial w}{\partial t}=&\eps^2f(W,t)+\eps^2f'(W,t)\cdot (w+\eps^2N_1+\eps^4N_2)\\
&+\eps^2\int_0^1f''(\alpha(w+\eps^2N_1+\eps^4N_2)+W)\cdot (w+\eps^2N_1+\eps^4N_2)\\
&\hphantom{XXXXXXXXXXXXXXXXXXX}\otimes(w+\eps^2N_1+\eps^4N_2)(1-\alpha)d\alpha\\
&-\eps^2 R_1(W)-\eps^4 R_2(W)-\eps^2\frac{\partial N_1}{\partial t}(W,t)-\eps^2N_1'(W,t)\cdot(\eps^2 R_1(W)+\eps^4 R_2(W))\\
&-\eps^4\frac{\partial N_2}{\partial t}(W,t)-\eps^4N_2'(W,t)\cdot(\eps^2 R_1(W)+\eps^4 R_2(W)).
\end{align*}

\noi
By the equations \eqref{eqn R,N}, it follows that
\begin{align*}
\frac{\partial w}{\partial t}=&\eps^2f'(W,t)\cdot (w+\eps^4N_2)\\
&+\eps^2\int_0^1f''(\alpha(w+\eps^2N_1+\eps^4N_2)+W)\cdot (w+\eps^2N_1+\eps^4N_2)\\
&\hphantom{XXXXXXXXXXXXXXXXXXX}\otimes(w+\eps^2N_1+\eps^4N_2)(1-\alpha)d\alpha\\
&-\eps^6N_1'(W,t)\cdot R_2(W)-\eps^4N_2'(W,t)\cdot(\eps^2 R_1(W)+\eps^4 R_2(W)).
\end{align*}

\noi
Integrating from $0$ to $t$, we then obtain that
\begin{align}\label{Duhamel difference ord 2}
w(t)=&\eps^2\int_0^tf'(W,\tau)\cdot w(\tau)d\tau+ \eps^6\int_0^t f'(W,\tau)\cdot N_2(W,\tau)d\tau \\
&-\eps^6\int_0^t N_1'(W,\tau)\cdot R_2(W)d\tau-\eps^4\int_0^tN_2'(W,\tau)\cdot(\eps^2 R_1(W)+\eps^4 R_2(W))d\tau\notag\\
&+\eps^2\int_0^t\int_0^1f''(\alpha(w+\eps^2N_1+\eps^4N_2)+W)
\cdot (w+\eps^2N_1+\eps^4N_2)\notag\\
&\hphantom{XXXXXXXXXXXXXXXXXXX}\otimes(w+\eps^2N_1+\eps^4N_2)(1-\alpha)d\alpha d\tau.\notag
\end{align}

\subsection{Study of the second order averaged equation in the case of $\T$}

Let $W_0\in H^s_+(\T)$, $s>1/2$. We
consider the averaged equation
\begin{equation*}
\begin{cases}
\partial_tW=\eps^2 R_1(W)+\eps^4 R_2(W)\\
W(0)=W_0.
\end{cases}
\end{equation*}

\noi
Since we already computed $R_1$ and $R_2$, we can rewrite this equation as:
\begin{equation*}
\begin{cases}
\partial_tW=\eps^2 f_{\textup{res}}(W)+\eps^4 \{f'(W,t)\cdot F_{\textup{osc}}(W,t)\}_{\textup{res}}\\
W(0)=W_0.
\end{cases}
\end{equation*}

\noi
Setting $\mathcal{W}=\eps W$, we have that $\mathcal{W}$ satisfies the equation:
\begin{equation}\label{averaging eqn}
\begin{cases}
\partial_t\mathcal{W}=f_{\textup{res}}(\W)+\{f'(\W,t)\cdot F_{\textup{osc}}(\W,t)\}_{\textup{res}}\\
\W(0)=\eps W_0=:\W_0.
\end{cases}
\end{equation}

\begin{lemma}\label{LWP for averaging eqn}
Let $s>\frac{1}{2}$. The problem \eqref{averaging eqn} is locally well-posed in $H^s(\T)$ at least on a time-interval
$[0,\frac{C}{\eps^2}]$, where $C>0$.
\end{lemma}

\begin{proof}
We first estimate the two terms on the right hand-side of equation \eqref{averaging eqn}.
By equation \eqref{f_res T}, we have that
\begin{align*}
\|R_1(\W)\|_{H^s(\T)}=\|f_{\textup{res}}(\W)\|_{H^s(\T)}\leq C\|\W\|_{H^s(\T)}^3.
\end{align*}

\noi
Then, we explicitly write the Fourier coefficients of $\{f'(\W,t)\cdot F_{\textup{osc}}(\W,t)\}_{\textup{res}}$. Since we have
\begin{align*}
\mathcal{F}\big(f(\W)\big)(k)=-i\sum_{k-l+m-j=0}e^{it(|k|-|l|+|m|-|j|)}\hat{\W}(j)\hat{\W}(l)\cj{\hat{\W}}(m),
\end{align*}

\noi
it follows that
\begin{align}\label{eqn f'.Fosc}
&\mathcal{F}\Big(f'(\W)\cdot F_{\textup{osc}}(\W,t)\Big)(k)\\
&=-2i\sum_{k-l+m-j=0}
e^{it(|k|-|l|+|m|-|j|)}\ft{F_{\textup{osc}}(\W,t)}(j)\hat{\W}(l)\cj{\hat{\W}}(m)\notag\\
&\hphantom{XX}-i\sum_{k-l+m-j=0}e^{it(|k|-|l|+|m|-|j|)}\hat{\W}(j)\hat{\W}(l)\cj{\ft{F_{\textup{osc}}(\W)}}(m)\notag\\
&=2i \sum_{\substack{k-l+m-j=0\\j-n+p-q=0\\|j|-|n|+|p|-|q|\neq 0}}
e^{it(|k|-|l|+|m|-|j|)}\frac{e^{it(|j|-|n|+|p|-|q|)}}{|j|-|n|+|p|-|q|}\hat{\W}(n)\hat{\W}(q)\cj{\hat{\W}}(p)\hat{\W}(l)\cj{\hat{\W}}(m)\notag\\
&\hphantom{XX}+i \sum_{\substack{k-l+m-j=0\\m-n+p-q=0\\|m|-|n|+|p|-|q|\neq 0}}
e^{it(|k|-|l|+|m|-|j|)}\hat{\W}(j)\hat{\W}(l)\frac{e^{-it(|m|-|n|+|p|-|q|)}}{|m|-|n|+|p|-|q|}
\cj{\hat{\W}}(n)\cj{\hat{\W}}(q)\hat{\W}(p).\notag
\end{align}

\noi
Then, $R_2(\W)$, the resonant part of $f'(\W,t)\cdot F_{\textup{osc}}(\W,t)$, has the following Fourier coefficients:
\begin{align}\label{eqn R_2}
\mathcal{F}(R_2(\W))=&\mathcal{F}\Big(\{f'(\W)\cdot F_{\textup{osc}}(\W,t)\}_{\textup{res}}\Big)(k)\\
=&2i \sum_{\substack{k-l+m-j=0\\j-n+p-q=0\\|j|-|n|+|p|-|q|\neq 0\\|k|-|l|+|m|-|n|+|p|-|q|=0}}
\frac{1}{|j|-|n|+|p|-|q|}\hat{\W}(n)\hat{\W}(q)\cj{\hat{\W}}(p)\hat{\W}(l)\cj{\hat{\W}}(m)\notag\\
&+i \sum_{\substack{k-l+m-j=0\\m-n+p-q=0\\|m|-|n|+|p|-|q|\neq 0\\|k|-|l|-|j|+|n|-|p|+|q|=0}}
\frac{1}{|m|-|n|+|p|-|q|}\hat{\W}(j)\hat{\W}(l)\cj{\hat{\W}}(n)\cj{\hat{\W}}(q)\hat{\W}(p)\notag
\end{align}

\noi
Noticing that
\begin{align*}
&|\mathcal{F}\Big(\{f'(\W)\cdot F_{\textup{osc}}(\W,t)\}_{\textup{res}}\Big)(k)|\leq2 \sum_{k-l+m-n+p-q=0}|
\hat{\W}(n)\hat{\W}(q)\cj{\hat{\W}}(p)\hat{\W}(l)\cj{\hat{\W}}(m)|\\
&\hphantom{XXXXXXXXXXXXXXXXX}+ \sum_{k-l-j+n-p+q=0}|\hat{\W}(j)\hat{\W}(l)\cj{\hat{\W}}(n)\cj{\hat{\W}}(q)\hat{\W}(p)|,
\end{align*}

\noi
and proceeding as in the proof of Lemma \ref{Est T 1}, we obtain
\begin{align}\label{ineq R_2}
&\|R_2(\W)\|_{H^s(\T)}=\|\{f'(\W)\cdot F_{\textup{osc}}(\W,t)\}_{\textup{res}}\|_{H^s}\leq C\|\W\|_{H^s(\T)}^5.
\end{align}

We use a standard fixed point argument to prove that equation \eqref{averaging eqn} is locally well-posed. Define
\begin{align*}
A\W(t):=\W(0)+\int_0^tf_{\textup{res}}(\W(\tau))d\tau+
\int_0^t\{f'(\W(\tau),\tau)\cdot F_{\textup{osc}}(\W(\tau),\tau)\}_{\textup{res}}d\tau.
\end{align*}

\noi
We intend to show that there is $T=\frac{C'}{\eps^2}$ such that $A$ is a contraction of the ball
\begin{align*}
B(R)=\Big\{\W\in C([0,T], H^s)\Big| \|\W\|_{L^{\infty}([0,T],H^s(\T))}\leq R\Big\},
\end{align*}

\noi
where $R=2\|\W_0\|_{H^s(\T)}=2\eps\|W_0\|_{H^s(\T)}$.
First we notice that $A$ acts on the ball $B(R)$. Indeed, let $\W\in B(R)$. Then,
\begin{align*}
\|A\W\|_{L^{\infty}([0,T],H^s(\T))}&\leq \|\W(0)\|_{H^s(\T)}+T\|f_{\textup{res}}(\W(\tau))\|_{L^{\infty}([0,T],H^s(\T))}\\
&+T\|\{f'(\W(\tau),\tau)\cdot F_{\textup{osc}}(\W(\tau),\tau)\}_{\textup{res}}\|_{L^{\infty}([0,T],H^s(\T))}\\
&\leq \|\W(0)\|_{H^s(\T)}+CT\|\W\|_{H^s(\T)}^3(1+\|\W\|_{H^s(\T)}^2)\\
&\leq \frac{R}{2}+CTR^3(1+R^2).
\end{align*}

\noi
Choosing $T=\frac{1}{2CR^2(1+R^2)}\leq\frac{C'}{\eps^2}$,
we obtain $\|A\W\|_{L^{\infty}([0,T],H^s(\T))}\leq R$
and thus \\ $A\W\in B(R)$. The fact that $A$ is a contraction follows similarly.
Therefore, there exists a unique solution
of equation \eqref{averaging eqn} in $B(R)$.
\end{proof}

\begin{proposition}\label{prop: W=Y second iteration}
Let $W_0\in H^s_+(\T)$, $s>1/2$.
The solution of the Cauchy problem \eqref{averaging eqn},
\begin{equation*}
\begin{cases}
\partial_t\mathcal{W}=f_{\textup{res}}(\W)+\{f'(\W,t)\cdot F_{\textup{osc}}(\W,t)\}_{\textup{res}}\\
\W(0)=\eps W_0=:\W_0,
\end{cases}
\end{equation*}

\noi
coincides with the solution of the following Cauchy problem:
\begin{equation}\label{eqn Y}
\begin{cases}
\partial_t\mathcal{Y}=-i\Pi_+(|\Y|^2\Y)-i\Pi_+(|\Y|^2\frac{1}{D}\Pi_
-(|\Y|^2\Y))-\frac{i}{2}\Pi_+(\Y^2\frac{1}{D}\cj{\Pi_-(|\Y|^2\Y)})\\
\Y(0)=\eps W_0=:\W_0
\end{cases}
\end{equation}

\noi
on its maximal interval of existence.
\end{proposition}

\begin{proof}
First we make the observation that we can easily prove local
well-posedness of equation \eqref{eqn Y} in $H^s(\T)$, $s>\frac{1}{2}$ on a time interval $[0,\frac{C}{\eps^2}]$,
following the lines of the proof of Lemma \ref{LWP for averaging eqn}.
Notice that $\Y\in L^2_+(\T)$. Therefore $\Y_-(t)=0$, for all $t$
in the maximal interval of existence of $\Y$.

In the following we prove that the only terms that do not contain $\W_-$ and thus,
contain only $\W_+$ in $\{f'(\W,t)\cdot F_{\textup{osc}}(\W,t)\}_{\textup{res}}$
are $-\Pi_+(|\W_+|^2\frac{1}{D}\Pi_-(|\W_+|^2\W_+))$ and $-\frac{1}{2}\Pi_+(\W_+^2\frac{1}{D}\cj{\Pi_-(|\W_+|^2{\W_+})})$.
Since all the other terms contain at least one factor $\W_-$, it results that the $\Y(t)=\Y_+(t)$
is also solution for equation \eqref{averaging eqn}. By Lemma \ref{LWP for
 averaging eqn}, we have uniqueness of the solution of equation \eqref{averaging eqn}.
Thus, $\Y$ is the unique solution of equation \eqref{averaging eqn}.

It is thus sufficient to determine the terms of
$\{f'(\W,t)\cdot F_{\textup{osc}}(\W,t)\}_{\textup{res}}$ which do not contain $\W_-$.
Let us consider the first term of the Fourier coefficient in equation \eqref{eqn R_2}:
\begin{align*}
2i \sum_{\substack{k-l+m-j=0\\j-n+p-q=0\\|j|-|n|+|p|-|q|\neq 0\\|k|-|l|+|m|-|n|+|p|-|q|=0}}
\frac{1}{|j|-|n|+|p|-|q|}\hat{\W}(n)\hat{\W}(q)\cj{\hat{\W}}(p)
\hat{\W}(l)\cj{\hat{\W}}(m)
\end{align*}

\noi
The first condition we have for the above sum is that $|j|-|n|+|p|-|q|\neq 0$. As we noticed in
Lemma \ref{Resonances T}, it follows that $j,n,p,q$ cannot be simultaneously non-positive or
non-negative, $j\neq n$, and $j\neq q$. Since in the above expression we have the factor
$\hat{\W}(n)\hat{\W}(q)\cj{\hat{\W}}(p)$, it follows that if we only want to have $\W_+$,
then the only possibility is $p,n,q\geq 0$ and $j<0$. In particular, this also satisfies $j\neq n$ and $j\neq q$.

The second condition we have for the above sum is $|k|-|l|+|m|-|n|+|p|-|q|=0$. As
a consequence, this yields $|k|-|l|+|m|-|j|=-(|j|-|n|+|p|-|q|)\neq 0$. Thus, $k,l,m,j$
cannot be simultaneously non-positive or non-negative, $k\neq l$, and $k\neq j$.
Since in the above sum we see appear the product $\hat{\W}(l)\cj{\hat{\W}}(m)$,
if we only want to have $\W_+$, it follows that we have two choices:
\begin{align*}
&\mathrm{(i)}\,\,\,\, k,l,m\geq 0; j<0 \text{ and } k\neq l,\\
&\mathrm{(ii)}\,\,\,\, k<0; l,m\geq 0; j<0.
\end{align*}

\noi
Note that if $k,l,m\geq 0, j<0$ and if $k=l$, then $k-l+m-n=0$ yields $m=j$,
which contradicts the fact that $m$ and $j$ have different signs. Thus, the
condition $k\neq l$ in $\mathrm{(i)}$ is redundant.

We compute $|k|-|l|+|m|-|n|+|p|-|q|$ for the second case $\mathrm{(ii)}$:
\begin{align*}
|k|-|l|+|m|-|n|+|p|-|q|&=-k-l+m-n+p-q\\
&=-2k+(k-l+m-j)+(j-n+p-q)=-2k<0.
\end{align*}

\noi
This contradicts the condition $|k|-|l|+|m|-|n|+|p|-|q|= 0$, and thus the case $\mathrm{(ii)}$
does not take place.

In the case $\mathrm{(i)}$, we have
\begin{align*}
|k|-|l|+|m|-|n|+|p|-|q|&=k-l+m-n+p-q\\
&=(k-l+m-j)+(j-n+p-q)=0.
\end{align*}

\noi
Moreover,
\begin{align*}
|j|-|n|+|p|-|q|&=-j-n+p-q=-2j+(j-n+p-q)=-2j=-2(n-p+q).
\end{align*}

\noi
Thus the only possible choice if we want to obtain terms that do not contain $W_-$, is the following:
\begin{align*}
&2i \sum_{\substack{k-l+m-n+p-q=0\\n-p+q<0\\k,l,m,n,p,q\geq 0}}\frac{1}{-2(n-p+q)}
\hat{\W}(n)\hat{\W}(q)\cj{\hat{\W}}(p)
\hat{\W}(l)\cj{\hat{\W}}(m)\\
&\hphantom{XXXXXXXXXXX}=-i\sum_{\substack{k-l+m-(n-p+q)=0\\n-p+q<0\\k,l,m,n,p,q\geq 0}}
\mathcal{F}\Big(\frac{1}{D}\Pi_-(|\W|^2\W)\Big)(n-p+q)
\hat{\W}(l)\cj{\hat{\W}}(m)\\
&\hphantom{XXXXXXXXXXX}=-i\mathcal{F}\Big(\Pi\big(|\W|^2\frac{1}{D}\Pi_-(|\W|^2\W)\big)\Big)(k).
\end{align*}

\noi
Proceeding similarly with the second resonant part in equation \eqref{eqn R_2}, which is equal to
\begin{align*}
i \sum_{\substack{k-l+m-j=0\\m-n+p-q=0\\|m|-|n|+|p|-|q|\neq 0\\|k|-|l|-|j|+|n|-|p|+|q|=0}}
\frac{1}{|m|-|n|+|p|-|q|}\hat{\W}(j)\hat{\W}(l)\cj{\hat{\W}}(n)\cj{\hat{\W}}(q)\hat{\W}(p),
\end{align*}

\noi
we obtain that it contains only one term in which $W_-$ does not appear, which is
\begin{align*}
-\frac{i}{2}\mathcal{F}\Big(\Pi_+\big(\W^2\frac{1}{D}\cj{\Pi_-(|\W|^2\W)}\big)\Big)(k).
\end{align*}

\noi
Therefore, the conclusion of the proposition follows.
\end{proof}

\begin{proposition}\label{lemma: comparison Y U}
Let $s>\frac{1}{2}$, $0\leq\alpha\leq\frac{1}{2}$, and $\delta>0$ small enough.
Consider the equations
\begin{equation*}
\begin{cases}
\partial_tY=-i\eps^2\Pi_+(|Y|^2Y)-i\eps^4\Pi_+(|Y|^2\frac{1}{D}\Pi_-(|Y|^2Y))
-\frac{i\eps^4}{2}\Pi_+(Y^2\frac{1}{D}\cj{\Pi_-(|Y|^2Y)})\\
Y(0)=W_0
\end{cases}
\end{equation*}

\noi
and
\begin{equation}\label{eqn U}
\begin{cases}
\partial_tU=-i\eps^2\Pi_+(|U|^2U)\\
U(0)=W_0.
\end{cases}
\end{equation}

\noi
Assume that
$\|U(t)\|_{H^s}\leq C\Big(\log(\frac{1}{\eps^{\delta}})\Big)^{\alpha}$ for all $t\in\R$.
Then, for $0\leq t\leq \frac{1}{\eps^2}\Big(\log(\frac{1}{\eps^{\delta}})\Big)^{1-2\alpha}$, we have that
\begin{align*}
\|Y(t)-U(t)\|_{H^s}\leq \eps^{2-C_0\delta},
\end{align*}

\noi
where $C_0>0$ is a constant and $\delta$ is chosen small enough such that $C_0\delta<1$.
In particular,
$\|Y(t)\|_{H^s}\leq C\Big(\log(\frac{1}{\eps^{\delta}})\Big)^{\alpha}$.
\end{proposition}

\begin{proof}
Set $Z:=Y-U$. Then $Z$ satisfies the equation
\begin{align*}
\begin{cases}
\partial_tZ=-i\eps^2\Big(\Pi_+(|Y|^2Y)-\Pi_+(|U|^2U)\Big)-i\eps^4\Pi_+(|Y|^2\frac{1}{D}
\Pi_-(|Y|^2Y))\\
\hphantom{XXX}-\frac{i\eps^4}{2}\Pi_+(Y^2\frac{1}{D}\cj{\Pi_-(|Y|^2Y)})\\
Z(0)=0.
\end{cases}
\end{align*}

\noi
We set also
\begin{align*}
h(U):&=-i\Pi_+(|U|^2U)\\
g(U):&=-i\Pi_+(|U|^2\frac{1}{D}\Pi_-(|U|^2U))-\frac{i}{2}\Pi_+(U^2\frac{1}{D}\cj{\Pi_-(|U|^2U)}).
\end{align*}

\noi
Then, we have
\begin{align*}
Z(t)=\eps^2\int_0^t\Big(h(Y(\tau))-h(U(\tau))\Big)d\tau+\eps^4\int_0^tg(Y(\tau))d\tau.
\end{align*}

\noi
Using the boundedness of the operators $\Pi_+$ and $\frac{1}{D}\Pi_-$ on $H^s(\T)$, and proceeding as in the proof
of Theorem \ref{Main theorem}, we obtain that
\begin{align*}
\|Z(t)\|_{H^s(\T)}\leq &\eps^2\Big(\log(\frac{1}{\eps^{\delta}})\Big)^{2\alpha}C\int_0^t\|Z(\tau)\|d\tau+\eps^4\Big(\log(\frac{1}{\eps^{\delta}})\Big)^{5\alpha}Ct.
\end{align*}

\noi
By Gronwall's inequality, it follows that for $0\leq t\leq \frac{1}{\eps^2}\Big(\log(\frac{1}{\eps^{\delta}})\Big)^{1-2\alpha}$ we have
\begin{align*}
\|Z(t)\|_{H^s(\T)}\leq & \eps^{2-C_0\delta},
\end{align*}

\noi
where $C_0>0$ is a constant.
\end{proof}

\subsection{Proof of Theorem \ref{Theorem second interation}}

\begin{lemma}\label{Est 2}
For $W\in H^s(\T)$ we have that
\begin{align*}
\|f'(W,t)\|\leq &C\|W\|_{H^s}^2,\\
\|f''(W,t)\|\leq &C\|W\|_{H^s}.
\end{align*}

\noi
where $\|\cdot\|$ denotes the operator norm of a bounded linear operator acting on $H^s(\T)$.
In addition, the following applications are continuous and $N$-linear on $H^s(\T)$:
\begin{enumerate}
\item $W\mapsto N_2(W,t)$ with $N=5$, \\
\item $W\mapsto f'(W,t)\cdot R_2(W)$, $W\mapsto N'_1(W,t)\cdot R_2(W)$, $W\mapsto N'_2(W,t)\cdot R_1(W)$
with $N=7$, \\
\item $W\mapsto N'_2(W,t)\cdot R_2(W)$ with $N=9$.
\end{enumerate}

\noi
In particular, if $\|W\|_{H^s(\T)}\leq \Big(\log(\frac{1}{\eps^{\delta}})\Big)^{\alpha}$, then their $H^s(\T)$-norms
are all bounded by $\Big(\log(\frac{1}{\eps^{\delta}})\Big)^{9\alpha}$.
\end{lemma}

\begin{proof}
The proof follows the same lines as that of Lemma \ref{Est T 1}.
\end{proof}

\begin{proof}[Proof of Theorem \ref{Theorem second interation}]
By Lemma \ref{prop: W=Y second iteration},
we have that the solution of the averaged equation \eqref{averaging eqn} is $\W(t)=\Y(t)$.
By hypothesis, we have that the solution
of equation \eqref{eqn U} satisfies $\|U(t)\|_{H^s(\T)}\leq C\Big(\log(\frac{1}{\eps^{\delta}})\Big)^{\alpha}$
Then, by Lemma \ref{lemma: comparison Y U},
it follows that $\|W(t)\|_{H^s(\T)}=\|Y(t)\|_{H^s(\T)}\leq C\Big(\log(\frac{1}{\eps^{\delta}})\Big)^{\alpha}$.
Using the estimates of Lemma \ref{Est 2}, it follows from equation \eqref{Duhamel difference ord 2}, that
\begin{align*}
\|w(t)\|_{H^s(\T)}\leq &\eps^2\Big(\log(\frac{1}{\eps^{\delta}})\Big)^{2\alpha}C\int_0^t\|w(\tau)\|_{H^s(\T)}d\tau+ \eps^6\Big(\log(\frac{1}{\eps^{\delta}})\Big)^{9\alpha}Ct,
\end{align*}

\noi
for $0\leq t\leq\frac{1}{\eps^2}\Big(\log(\frac{1}{\eps^{\delta}})\Big)^{1-2\alpha}$. Then, by Gronwall's inequality we obtain
\begin{align*}
\|w(t)\|_{H^s(\T)}\leq &\eps^6\Big(\log(\frac{1}{\eps^{\delta}})\Big)^{9\alpha}Cte^{\eps^2\Big(\log(\frac{1}{\eps^{\delta}})\Big)^{2\alpha}Ct}\\
\leq &\eps^4\Big(\log(\frac{1}{\eps^{\delta}})\Big)^{1+7\alpha}e^{C\log(\frac{1}{\eps^{\delta}})}\leq \eps^{4-C_0\delta},
\end{align*}

\noi
where $C_0>0$.
Thus, $\|u(t)-W(t)-\eps^2N_1(W,t)-\eps^4 N_2(W,t)\|_{H^s}\leq \eps^{4-C_0\delta}$ for $0\leq t\leq \frac{1}{\eps^2}\Big(\log(\frac{1}{\eps^{\delta}})\Big)^{1-2\alpha}$.
Since $\|N_2(W,t)\|_{H^s(\T)}\leq C\Big(\log(\frac{1}{\eps^{\delta}})\Big)^{5\alpha}$, this yields
\[\|u(t)-W(t)-\eps^2F_{\textup{osc}}(W,t)\|_{H^s(\T)}=\|u(t)-W(t)-\eps^2N_1(W,t)\|_{H^s(\T)}\leq \eps^{4-C_0\delta}.\]

\noi
Changing back to the variables $v=\eps e^{-i|D|t}u$ and $\W=\eps W$, the conclusion of the theorem follows:
\[\|v(t)-e^{-i|D|t}\big(\W(t)+F_{\textup{osc}}(\W,t)\big)\|_{H^s(\T)}=\|u(t)-W(t)-\eps^2N_1(W,t)\|_{H^s(\T)}\leq \eps^{5-C_0\delta}.\]
\end{proof}

\noi
{\bf Acknowledgments:}
The author would like to thank her Ph.D. advisor Prof. Patrick G\'erard
for suggesting this problem to her and for interesting discussions. She is also grateful to Tadahiro Oh
for giving her the reference \cite{Abou Salem}, from which she learned about the 
renormalization group method.


\begin{thebibliography}{99}

\bibitem{Abou Salem} W.K. Abou Salem,
{\it On the renormalization group approach to perturbation theory for PDEs,}
Ann. Henri Poincar\'e 11, no. 6, 1007--1021 (2010).

\bibitem{Bogolyubov and Mitropol'skii} N.N. Bogolyubov, Yu. A. Mitropol'skii
{\it Asymptotic Methods in the Theory of Nonlinear Oscillations,}
Hindustan, Delhi (1958).

\bibitem{BGT} N. Burq, P. G\'erard, N. Tzvetkov,
{\it Bilinear eigenfunction estimates and the nonlinear Schr\"odinger equation 
on surfaces,} Invent. Math., 159 (2005), 187--223.

\bibitem{CGO 1994} L.-Y. Chen, N. Goldenfeld, Y. Oono,
{\it Renormalization group theory for global
asymptotic analysis,}
Phys. Rev. Lett. 73(10), 1311--1315 (1994).

\bibitem{CGO 1996} L.-Y. Chen, N. Goldenfeld, Y. Oono,
{\it Renormalization group and singular perturbations:
multiple scales, boundary layers, and reductive perturbation theory,}
Phys. Rev. E 543(1), 376--394 (1996).

\bibitem{I team} J. Colliander, M.Keel, G. Staffilani, H. Takaoka, T. Tao,
{\it Transfer of energy to high frequencies in the cubic defocusing nonlinear Schrödinger equation,}
 Invent. Math. 181 (2010), no. 1, 39--113.

\bibitem{De Ville} R. De Ville, A. Harkin, M. Holzer, K. Josic, T. Kaper,
{\it Analysis of a renormalization
group method and normal form theory for perturbed ordinary differential
equations,}
Physica D 237, 1029--1052 (2008).

\bibitem{PGSG} P. G\'erard, S. Grellier,
{\it The cubic Szeg\"{o} equation,} Annales Scientifiques de l'Ecole 
Normale Sup\'erieure, Paris, $4^e$ s\'erie, t. 43, (2010), 761--810.

\bibitem{PGSGIT} P. G\'erard, S. Grellier,
{\it Invariant tori for the cubic Szeg\"{o} equation,} preprint \url{arXiv:1011.5479v1}, to appear in Inventiones Mathematicae.

\bibitem{PGSG res} P. G\'erard, S. Grellier,
{\it Effective integrable dynamics for some nonlinear wave equation}, \url{arXiv:1110.5719v1}.

\bibitem{Germain 1} P. Germain,
{\it Space-time resonances}, 	
arXiv:1102.1695.

\bibitem{Germain 2} P. Germain, N. Masmoudi, J. Shatah,
{\it Global solutions for 2D quadratic Schr\"odinger equations,}
arXiv:1001.5158.

\bibitem{Germain 3} P. Germain,
{\it Global existence for coupled Klein-Gordon equations with different speeds,}
	arXiv:1005.5238.

\bibitem{Germain 4} P. Germain, N. Masmoudi, J. Shatah,
{\it Global solutions for the gravity water waves equation in dimension 3,}
C. R. Math. Acad. Sci. Paris 347 (2009), no. 15--16, 897--902.

\bibitem{Germain 5} P. Germain, N. Masmoudi, J. Shatah,
{\it Global solutions for 3D quadratic Schrödinger equations,}
Int. Math. Res. Not. IMRN 2009, no. 3, 414--432.

\bibitem{GNT} S. Gustafson, K. Nakanishi, T.-P. Tsai,
{\it Scattering theory for the Gross-Pitaevskii equation in three dimensions,}
Commun. Contemp. Math. 11 (2009), no. 4, 657--707.


\bibitem{Lax} P. Lax, {\it Integral of nonlinear equations of evolution and solitary waves,}
Comm. Pure and Applied Math., 101 (1968), 467--490.

\bibitem{Moise Temam} I. Moise, R. Temam,
{\it Renormalization group method. Applications to Navier–
Stokes equation,}
Discret. Continuous Dyn. Syst. 6, 191--200 (2000).

\bibitem{Moise Ziane} I. Moise, M. Ziane,
{\it Renormalization Group Method. Applications to Partial
Differential Equations,}
J. Dyn. Differ. Equ. 13, 275--321 (2001).

\bibitem{Shatah} J. Shatah,
{\it Space-time resonances,}
Quart. Appl. Math. 68 (2010), no. 1, 161--167.

\bibitem{Petcu Temam}
M. Petcu, R. Temam, D. Wirosoetisno,
{\it Renormalization group method applied
to the primitive equations,}
J. Differ. Equ. 208, 215--257 (2005).

\bibitem{pocov1} O. Pocovnicu,
{\it Traveling waves for the cubic Szeg\"{o} equation on the real line,} to appear in Analysis and PDE.

\bibitem{pocov2}
O. Pocovnicu,
{\it Explicit formula for the solution of the Szeg\"o equation on the real line and applications,}
Disc. Cont. Dyn. Sys.-A, Vol. 31, no. 3, (2011), 607-649.

\bibitem{Temam Wirosoetisno} R. Temam, D. Wirosoetisno,
{\it Averaging of differential equations generating oscillations and an application to control,}
Special issue dedicated to the memory of Jacques-Louis Lions. Appl. Math. Optim. 46, no. 2-3, 313--330 (2002).

\bibitem{Ziane} M. Ziane,
{\it On a certain renormalization group method,}
J.Maths.Phys., 41 (5), (2000).

\end{thebibliography}
\end{document}